\documentclass[10pt]{article}
\usepackage[english]{babel}
\usepackage[dvips]{color}
\usepackage{graphicx, amsmath, amssymb, amsfonts, amsthm}
\usepackage{slashbox}
\usepackage{color}

\newtheorem{theorem}{Theorem}[section]

\newtheorem{lemma}[theorem]{Lemma}

\renewcommand{\arraystretch}{1.1}

\makeatletter   
\@addtoreset{equation}{section}

\makeatother

\newcommand{\refpart}[1]{{\it (#1)}}                 
\newcommand{\ppt}{\hspace{1pt}}                   

\newcommand{\heun}[5]{\mbox{\rm Hn}\!\left( {#1 \atop #2} \left|  {#3   \atop #4} \right|\, #5 \right) }
\newcommand{\heunoppa}[5]{\mbox{\rm Hn}(#1,#2\,|\,#3;\,#4\,|\,#5)}
\newcommand{\hpgoppa}[5]{{}_{#1}{\rm F}_{#2}(#3;#4\,|\,#5)}
\newcommand{\heuno}{\mbox{\rm Hn}}
\newcommand{\hpgo}[2]{{}_{#1}\mbox{\rm F}_{\!#2}}
\newcommand{\hpg}[5]{{}_{#1}\mbox{\rm F}_{\!#2}\!
  \left(\left.{#3 \atop #4}\right|\, #5 \right) }

\newcommand{\led}[1]{#1}                                    
\newcommand{\hpgde}[1]{E(#1)}                         
\newcommand{\heunde}[1]{\mbox{\it HE\ppt}(#1)}  

\newcommand{\branch}[3]{#1\,=\,#2\,=\,#3}      
\newcommand{\brep}[2]{$[#1]_{#2}\ppt$}              

\newcommand{\comp}[2]{${#2}\cdot{#1}$}
\newcommand{\compp}[3]{${#3}\cdot{#2}\cdot{#1}$}
\newcommand{\nocomposition}{\mbox{\rm indecomposable}}

\newcommand{\degr}{D}    
\newcommand{\la}{\alpha}                                     
\newcommand{\lb}{\beta}
\newcommand{\lc}{\gamma}
\newcommand{\ld}{\delta}

\newcommand{\PR}[1]{\mbox{\rm P#1}}       
\newcommand{\PH}[1]{\mbox{$H_{#1}$}}        

\newcommand{\GT}[1]{#1}       
\newcommand{\HT}[1]{{#1}_H}   
\newcommand{\CT}[1]{{#1}_C}   
\newcommand{\AT}[1]{{#1}_A}   
\newcommand{\BT}[1]{{#1}_B}   
\newcommand{\DDT}{2\!\times\!2}  
\newcommand{\pback}[1]{\stackrel{\ppt#1}{\longleftarrow}}   

\newcommand{\A}{A}
\newcommand{\B}{B}
\newcommand{\C}{C}

\newcommand{\PP}{{\Bbb P}}
\newcommand{\ZZ}{{\Bbb Z}}
\newcommand{\CC}{{\Bbb C}}

\newcommand{\QQ}{{\Bbb Q}}

\newcommand{\equal}{&\!\!\!=\!\!\!&}
\newcommand{\fr}[2]{#1 / \,#2}

\begin{document}


\title{Parametric transformations between the Heun \\ and Gauss hypergeometric functions}

\author{Raimundas Vidunas\footnote{Faculty of Mathematics, Kobe University, Rokko-dai 1-1, Nada-ku,
        657-8501 Kobe, Japan. E-mail: vidunas@math.kobe-u.ac.jp},
        Galina Filipuk\footnote{Faculty of Mathematics, Informatics, and Mechanics, University of
        Warsaw, Banacha 2, 02-097 Warsaw, Poland. E-mail:
        filipuk@mimuw.edu.pl}}
\date{}
\maketitle

\begin{abstract}
The hypergeometric and Heun functions are classical special functions.
Transformation formulas between them are commonly induced by pull-back transformations
of their differential equations, with respect to some coverings $\PP^1\to\PP^1$.  
This gives expressions of Heun functions in terms of better understood hypergeometric functions.
This article presents the list of hypergeometric-to-Heun pull-back transformations with a
free continuous parameter, and illustrates most of them by a Heun-to-hypergeometric reduction formula.
In total, 61 parametric transformations exist, of maximal degree~12.
\end{abstract}

\section{Introduction}

The Gauss hypergeometric function $\hpgoppa{2}{1}{a,b}{c}{z}$
and the local Heun function $\heunoppa{t}{q}{a,b}{c,d}{x}$ 
are classical special functions, holomorphic in a neighborhood of $z=0$, respectively $x=0$. 
They are solutions of canonical second-order Fuchsian differential equations 
on the Riemann sphere~$\PP^1$, having $3$ or~$4$ regular singular points, respectively.  
The Fuchsian equations are the \emph{Gauss hypergeometric equation} 
and the {\em Heun equation}~\cite{BE}.
We present these equations and functions soon, in formulas
(\ref{HGE})--(\ref{eq:HF}).

The special functions $\hpgo{2}{1}$ and~$\heuno$ satisfy many identities
such as 
\begin{align}
  \label{eq:first}
  \hpg{2}{1}{2a,\,2b\,}{a+b+\tfrac12}{\,x}
  &=\hpg{2}{1}{a,\,b\,}{a+b+\tfrac12}{\,4x(1-x)}, 
\\  \label{eq:second}
 \heun{{\tfrac12}}{2ab}{2a,\,2b}{c,\;c
 }{x}&=\hpg{2}{1}{a,\,b}{c}{\,4x(1-x)}. 
\end{align}
The former is Gauss' quadratic transformation of~$\hpgo{2}{1}$, 
and the latter is a well-known generalization with an extra free parameter.  
It can be viewed as
a \emph{Heun-to-hypergeometric}, or \emph{Heun-to-Gauss, reduction}.

Transformations between hypergeometric functions were first systematically investigated by
Goursat~\cite{Goursat1881}.  A~complete classification, with a few sets of unpredicted
transformations, was recently performed by the first author~\cite{VidunasFE}.
Several Heun-to-Gauss reduction formulas we found my Maier~\cite{Maier}.
These transformations are based on a rational map
$z=\varphi(x)$, such as $\varphi(x)=4x(1-x)$. 

Generally, the considered transformations are induced by {\em pull-back transformations}
of Fuchsian equations 
of the form
\begin{equation} \label{eq:algtransf}
z\longmapsto\varphi(x), \qquad y(z)\longmapsto
Y(x)=\theta(x)\,y(\varphi(x)),
\end{equation}
Here $\varphi(x)$ is a rational function representing the pullback covering
along which a hypergeometric (or more general Fuchsian) equation is \emph{lifted}
or \emph{pulled back}. The \emph{gauge prefactor} $\theta(x)$ is a radical
function (i.e., a product of powers of rational functions). It is usually chosen 
to yield a pulled-back equation with fewer singularities and standard values of local exponents.
The \emph{degree} of a pull-back transformation is the degree of $\varphi(x)$.
The considered  {\em hypergeometric-to-Heun} pull-back transformations
will be called  {\em Gauss-to-Heun} transformations (or pull-backs) for brevity.
We encounter {\em Heun-to-Heun} and  {\em Gauss-to-Gauss} (or just {\em hypergeometric}) 
transformations as well.

This article 
focuses on the Heun-to-hypergeometric reductions
\begin{equation}
\label{eq:H-2F1 gen}
 \heun{t}{q}{a,\,b}{c,\,d}{x}=\theta(x)\,\hpg{2}{1}{\A,\,\B}{\C}{\varphi(x)}.
\end{equation}
induced by the pull-back coverings from hypergeometric to Heun equations, with at least 
one free parameter. Quadratic transformations such as (\ref{eq:second}) were first indicated 
by Kuiken \cite{Kuiken}. The several transformations of Maier~\cite{Maier} 
are all parametric Heun-to-Gauss 
reductions without the prefactor $\theta(x)$. 

In a parallel article \cite{HeunClass}, the authors 
classify the pull-back coverings appearing in parametric Heun-to-Gauss reductions.
The present article spells out the list of parametric Heun equations reducible to the hypergeometric
ones via pull-back transformations (\ref{eq:algtransf}), and gives an extensive list of
transformation formulas between the Heun and hypergeometric functions.
Up to M\"obius fractional-linear transformations, there are 61 different cases
of Heun-to-Gauss reductions (excluding infinite families of transformations
from hypergeometric equations with cyclic or a dihedral monodromy).
Among these 61 Heun-to-Gauss reductions, 28 are compositions of lower degree 
transformations among Heun and hypergeometric functions.
The maximal degree of a parametric Heun-to-Gauss reduction is 12.

The coverings that occur in these pull-back transformations turn~out to be \emph{Belyi maps}, 
in the sense that they have at~most three critical values on~$\PP_z^1$.  The coverings characteristically
branch only above the singular points $z=0$, $z=1$, $z=\infty$ of the hypergeometric equation.
The four singular points of Heun's equation  lie above the same set
$\{0,\,1,\,\infty\}\subset \PP_z^1$. 
In the 61 Heun-to-hypergeometric reductions, 48 different Belyi coverings are involved.
Herfurtner's list  \cite{Herfurtner91} of elliptic surfaces with 4 singular fibers contains
38 of these coverings as 
Klein's $\cal J$-invariants of the elliptic surfaces.

This article along with  \cite{HeunClass} skips the degenerate cases of parametric
pull-back transformations from the hypergeometric equations with cyclic or dihedral monodromy.
There are pull-backs to Heun equations of any degree from
these hypergeometric equations, as presented in \cite{VidunasHDD}. Morover,
non-Belyi coverings can occur in parametric pull-backs from dihedral hypergeometric equations
\cite[Proposition 2.3]{HeunClass}.
Together with van Hoeij, the first author has already started to classify Heun-to-Gauss reductions
without any free parameter in the so-called {\em hyperbolic case} \cite{VidunasHoeij}.

Here is the content and the structure of the article.
Section \ref{belyiclass} recalls the list of 61 Heun-to-Gauss reductions
(up to M\"obius transformations) obtained in \cite{HeunClass}. 
In \S \ref{heunclass} this list is rewritten in an order convenient for answering the following basic question:
\begin{quote}
Is a given Heun function with a free parameter reducible to Gauss hypergeometric functions
by a pull-back transformation?
\end{quote}
The newly ordered Heun equations have the labels P1 to P61.
In \S \ref{tvalues} we inspect the $t$-values of the reducible Heun functions
and make an arithmetic observation.
Section \ref{heunhpgid} explains how to obtain identities between Heun and hypergeometric
functions induced by the listed pull-back transformations.
Section \ref{sec:hpghe} is a comprehensive survey on Heun-to-Gauss reductions,
including brief overviews of Gauss-to-Gauss and Heun-to-Heun transformations.
Only some less interesting one-parameter composite Heun-to-Gauss reductions 
are not exemplified by a formula.
Appendix A reminds the symmetries of hypergeometric and Heun equations.
Appendix B reviews the composite transformations among the Heun-to-Gauss reductions.
Appendix C gives additional invariants (of the fractional-linear transformations) to identify 
the reducible Heun equations completely.

Before presenting the transformation lists and formulas,
we briefly recall that the hypergeometric and Heun equations are 
\begin{equation} \label{HGE}
\frac{d^2y(z)}{dz^2}+
\left(\frac{\C}{z}+\frac{\A+\B-\C+1}{z-1}\right)\,\frac{dy(z)}{dz}+\frac{\A\,\B}{z\,(z-1)}\,y(z)=0
\end{equation}
and, respectively,
\begin{equation}\label{Heun}
\frac{d^2y(x)}{dx^2}+\biggl(\frac{c}{x}+\frac{d}{x-1}+\frac{a+b-c-d+1}{x-t}\biggr)\frac{dy(x)}{dx}
+\frac{abx-q}{x(x-1)(x-t)}y(x)=0.
\end{equation}
They are canonical second-order Fuchsian differential equations on the Riemann sphere~$\PP^1$,
having $3$ and~$4$ regular singular points respectively \cite{BE}. 
Any 
second order Fuchsian 
equation with $3$ or $4$ 
singularities can be transformed to them by M\"obius transformations.
The singular points of these equations are $z=0$, $z=1$, $z=\infty$ and
$x=0$, $x=1$, $x=t$, $z=\infty$. The information about the singularities and local
exponents is encoded in the Riemann $P$-schemes for these equations:
\begin{equation} \label{eq:P}
P\left\{\begin{array}{ccc} 0 & 1 & \infty
\\ 0 & 0 & \A \\ 1-\C & \C-\A-\B & \B
\end{array} \;z\; \right\}, \qquad
P\left\{\begin{array}{cccc} 0 & 1 &t & \infty
\\ 0 & 0 &0 & a \\ 1-c & 1-d & c+d-a-b & b
\end{array} \;x\; \right\},
\end{equation}
so that the local exponents at $z=0$ for the hypergeometric equation are $0,1-C$, etc.
Recall that Fuchsian equations with 3 singularities are defined
uniquely by their singularities and the local exponents. This is not
generally true for Fuchsian equations with more singularities.
In particular, $q$ is an {\em accessory parameter} of Heun's equation.

The local solution at $z=0$ with the local exponent $0$ and the value $1$
of the hypergeometric equation is the well-known Gauss hypergeometric series:
\begin{equation} \label{gauss2f1}
\hpg{2}{1}{\A,\,\B\,}{\C}{\,z}= 
\sum_{n=0}^{\infty}\frac{(\A)_n\,(\B)_n}{(\C)_n\,n!}\,z^n.
\end{equation}
The local solution at $x=0$ with the local exponent $0$ and the value $1$
of Heun's equation is denoted by
\begin{equation} \label{eq:HF}
\heun{\,t\,}{q}{\,a,\,b\,}{c;\,d}{\,x\,}.
\end{equation}
The power series $\sum_{n=0}^{\infty} h_n x^n$ for this \emph{(local) Heun function} is
rather complicated. Its coefficients $h_n$ satisfy a  second order linear
recurrence relation \cite{MaierPH},  with the coefficients quadratic in $n$.
Provided that $c$ is not a non-positive integer  (to avoid division by zero),
$\heuno(x)$ is defined and holomorphic at least on $|x|<\min(1,|t|)$.
The Heun function degenerates to the $\hpgo21$ function if $d=a+b-c+1$ and $q=a b t$.
Notice the corresponding degeneration of Heun's equation to (\ref{HGE}). If $q\neq a b t$,
the point $x=t$ becomes logarithmic rather than ordinary.
The Heun function is identical to the constant 1 if $ab=0$
and $q=0$. Note that the parameters $a, b$ are symmetric and give
the local exponents at $x=\infty$, whereas the parameters $c$ and
$d$ are not interchangeable and determine the non-zero local
exponents at $x=0$ and $x=1$, respectively.

The Heun equation contains a large number of interesting special cases.
In particular, the Lam\'e equation \cite{BE} is the most studied case
(it is basically $a+b=c=d=1/2$)
and has considerable importance in mathematical physics \cite{BE, Ronveaux}.
The Heun equation appears in problems such as
diffusion, wave propagation, magneto-hydrodynamics, heat and mass
transfer, particle physics and cosmology of the very early universe.
Heun functions are much less  understood than the hypergeometric
functions. Particularly, no general elementary integral representation of
Heun functions is known. It is thus generally desirable to have expressions of
Heun's and especially Lam\'e functions in terms of more elementary functions.

\section{Two classifications of Heun-to-hypergeometric reductions}
\label{belyiclass}

In this section we first recall the classification of Heun-to-Gauss reductions in \cite{HeunClass}
up to M\"obius transformations,
then we rewrite the list in the order convenient for finding out  whether an encountered
Heun function or equation is reducible to a hypergeometric one. An arithmetic observation
on the $t$-parameters of the reducible Heun functions is made in \S \ref{tvalues}.

First we introduce some notation. From (\ref{eq:P}) it is clear that the local exponent differences
of the hypergeometric equation at the singular points are $1-C$, $C-A-B$, $A-B$, while the 
exponent differences of Heun's equation are $1-c$, $1-d$, $c+d-a-b$, $a-b$.
As in \cite{HeunClass}, let $\hpgde{\la,\lb,\lc}$ denote a hypergeometric equation with the 
exponent differences $\la,\lb,\lc$, and let $\heunde{\la,\lb,\lc,\ld}$
denote a Heun equation with the 
exponent differences $\la,\lb,\lc,\ld$.

A pull-back transformation of degree $\degr$ from a hypergeometric equation
$\hpgde{\la_1,\lb_1,\lc_1}$ to Heun's equation $\heunde{\la_2,\lb_2,\lc_2,\ld_2}$
is denoted by $\hpgde{\la_1,\lb_1,\lc_1}\pback{\degr}\heunde{\la_2,\lb_2,\lc_2,\ld_2}$.
Sometimes we indicate the pullback covering more specifically by putting
a subscript to the degree $D$. Similarly, a pull-back between hypergeometric equations 
is denoted by $\hpgde{\la_1,\lb_1,\lc_1}\pback{\degr}\hpgde{\la_2,\lb_2,\lc_2}$.
Like in the similar notation $(\la_1,\lb_1,\lc_1)\pback{\degr}(\la_2,\lb_2,\lc_2)$ in \cite{VidunasFE},
the arrow follows the direction of the covering $\varphi:\PP_x^1\to\PP^1_z$.
Our notation points to existence of differential equations with the given 
exponent differences and related by an indicated pull-back transformation, rather 
than to existence of pull-backs 
between any differential equations with the given 
exponent differences. The order of 
exponent differences on both sides of the arrow is irrelevant for us,
as we do not assign them to particular singularities by this notation.

From now on, let $\omega$ denote the cubic root of unity $\exp(2\pi i/3)$.

\subsection{The starting classification}

The parallel paper \cite{HeunClass} classifies the pull-back transformations (\ref{eq:algtransf})
between the hypergeometric and Heun equations with a free continuous parameter 
up to M\"obius transformations. We recall the results in Tables \ref{clasfig} and \ref{clasfig2} to be self-contained.

The classification in  \cite{HeunClass} starts with the hypergeometric equations with a free parameter
that could be pulled-back to Heun equations. To get a pulled-back equation with just 4 singular points,
some of the local exponent differences must be restricted to the value $1/k$, with $k$ a positive integer.
Since we want a free parameter, at least one 
exponent difference is left unrestricted. The quadratic transformation illustrated in (\ref{eq:second}) 
has no restrictions on the parameters of the hypergeometric equation, while other pull-back 
transformations have the following sequences of restricted exponent differences:
\begin{equation*} 
\left(\frac12\right), \; \left(\frac13\right), \; \left(\frac12,\frac13\right), \; \left(\frac12,\frac14\right), \;
\left(\frac12,\frac15\right), \; \left(\frac12,\frac16\right), \; \left(\frac13,\frac13\right), \;
\left(\frac13,\frac14\right), \; \left(\frac14,\frac14\right).
\end{equation*}
The classification in \cite{HeunClass} skips the restrictions $(1)$ and $(1/2,1/2)$ as they give
the rather degenerate cases of hypergeometric equations with cyclic or a dihedral monodromy.
There are infinitely many pull-backs to the Heun equations in this case,
as described in \cite{VidunasHDD}.

\begin{table} \small
\begin{center}
\begin{tabular}{@{}llclll@{}} 
\hline \multicolumn{2}{c}{Local exponent differences} & 
$\degr$ & Branching pattern 
& The covering, 
 & Characterization 
 \\ \cline{1-2} hypergeom.
& Heun's & 
& above singular points & its composition  & of \S \ref{heunclass} 
\\
\hline
$\led{\la,\,\lb,\,\lc}$ & $\led{\la,\,\la,\,2\lb,\,2\lc}$ & 2
& \branch{1+1}{2}{2} & \PH{32}, \nocomposition
& \PR{1}, $j=1728$ 
\\ $\led{\fr12,\,\la,\,\lb}$ & $\led{1/2,\,\la,\,2\la,\,3\lb}$ & 3
& \branch{\brep21+1}{2+1}{3} & \PH{34}, \nocomposition
& \PR{15}, $t=-3$ 
\\ 
& $\led{\la,\,\la,\,2\la,\,4\lb}$ & 4
& \branch{\brep22}{2+1+1}{4} & \PH{35},  \comp{\GT2}{2}
& \PR{3}, $j=1728$ 
\\ 
& $\led{\la,\,3\la,\,\lb,\,3\lb}$ & 
& \branch{\brep22}{3+1}{3+1 & \PH{47},  \nocomposition}
& \PR{20}, $t=-8$ 
\\ 
& $\led{2\la,\,2\la,\,2\lb,\,2\lb}$ & 
& \branch{\brep22}{2+2}{2+2} &  \PH{31}, $2\times 2$
& \PR{2}, $j=1728$ 
\\ 
$\led{1/3,\,\la,\,\lb}$ & $\led{\la,\,2\la,\,\lb,\,2\lb}$ & 3
& \branch{\brep31}{2+1}{2+1} & \PH{34},  \nocomposition
& \PR{19}, $t=-8$ 
\\ 
& $\led{\la,\,\la,\,\la,\,3\lb}$ & 
& \branch{\brep31}{3}{1+1+1} & \PH{33}, \nocomposition
& \PR{51}, $j=0$ 
\\ 
$\led{\fr12,\fr13,\la}$ & $\led{\fr12,\fr12,\fr13,4\la}$ & 4
& \branch{\brep21+1+1}{\brep31+1}{4} 
& \PH{36}, \nocomposition & \PR{47}, $t\in\QQ(\sqrt{-2})$ 
\\ 
& $\led{\fr12,\fr23,\,\la,\,4\la}$ & 5
& \branch{\brep22+1}{\brep31+2}{4+1} & \PH{29}, \nocomposition
& \PR{31}, $t=32/5$ 
\\ 
& $\led{\fr12,\fr23,\,2\la,\,3\la}$ & 
& \branch{\brep22+1}{\brep31+2}{3+2} &  \PH{30},   \nocomposition
& \PR{25}, $t=-4$ 
\\ 
&  $\led{\fr12,\fr13,\fr13,5\la}$ & 
& \branch{\brep22+1}{\brep31+1+1}{5} &  \PH{37}, \nocomposition
& \PR{59}, $t\in\QQ(\!\sqrt{-15})$ 
\\ 
& $\led{\fr12,\fr12,\,\la,\,5\la}$ & 6
& \branch{\brep22+1+1}{\brep32}{5+1} & \PH{26}, \nocomposition
& \PR{45}, $t\in\QQ(i)$ 
\\ 
&  $\led{\fr12,\fr12,\,2\la,\,4\la}$ & 
& \branch{\brep22+1+1}{\brep32}{4+2} &  \PH{27},  \comp{\GT3}{2}
& \PR{4}, $j=1728$ 
\\ 
&  $\led{\fr12,\fr12,\,3\la,\,3\la}$ & 
& \branch{\brep22+1+1}{\brep32}{3+3} & \PH{28}, \comp{\CT3}{\HT2}
& \PR{38}, $t\in\QQ(\sqrt{3})$ 
\\ 
& $\led{\fr13,\fr23,\,\la,\,5\la}$ & 
& \branch{\brep23}{\brep31+2+1}{5+1} & \PH{24},   \nocomposition
& \PR{26}, $t=25/9$ 
\\ 
& $\led{\fr13,\fr23,\,2\la,\,4\la}$ & 
& \branch{\brep23}{\brep31+2+1}{4+2} & \PH{25},   \comp{\GT2}{3}
& \PR{21}, $t=-8$ 
\\ 
& $\led{\fr13,\fr13,\fr13,6\la}$ & 
& \branch{\brep23}{\brep31+1+1+1}{6} &  \PH{38}, \comp{\GT2}{\CT3}
& \PR{52}, $j=0$ 
\\ 
& $\led{\fr12,\fr13,\,\la,\,6\la}$ & 7
& \branch{\brep23+1}{\brep32+1}{6+1} & \PH{21}, \nocomposition
& \PR{61}, $j\in\QQ(\omega)$ 
\\ 
& $\led{\fr12,\fr13,\,2\la,\,5\la}$ & 
& \branch{\brep23+1}{\brep32+1}{5+2} & \PH{22},   \nocomposition
& \PR{36}, $t=189/64$ 
\\ 
& $\led{\fr12,\fr13,\,3\la,\,4\la}$ & 
& \branch{\brep23+1}{\brep32+1}{4+3} & \PH{23},   \nocomposition
& \PR{30}, $t=-27$ 
\\ 
& $\led{\fr23,\,\la,\,\la,\,6\la}$ & 8 &
\branch{\brep24}{\brep32+2}{6+1+1} & \PH{15},   \comp{\GT4}{2}
& \PR{7}, $j=1728$ 
\\ 
& $\led{\fr23,\,\la,\,2\la,\,5\la}$ & 
&  \branch{\brep24}{\brep32+2}{5+2+1} &\PH{16},   \nocomposition
& \PR{29}, $t=27/2$ 
\\ 
& $\led{\fr23,\,2\la,\,3\la,\,3\la}$ & 
& \branch{\brep24}{\brep32+2}{3+3+2} & \PH{17}, \comp{\GT4}{2}
& \PR{8}, $j=1728$ 
\\ 
& $\led{\fr13,\fr13,\,\la,\,7\la}$ & 
& \branch{\brep24}{\brep32+1+1}{7+1} & \PH{18}, \nocomposition
& \PR{56}, $t\in\QQ(\omega)$ 
\\ 
& $\led{\fr13,\fr13,\,2\la,\,6\la}$ & 
& \branch{\brep24}{\brep32+1+1}{6+2} & \PH{19}, \comp{\GT2}{\BT4}, \comp{\GT4}{2}
& \PR{6}, $j=1728$ 
\\ 
& $\led{\fr13,\fr13,\,4\la,\,4\la}$ & 
& \branch{\brep24}{\brep32+1+1}{4+4} & \PH{20}, \comp{\GT2}{4}, \comp{\AT4}{\HT2}
& \PR{41}, $t\in\QQ(\sqrt{3})$ 
\\ 
& $\led{\fr12,\,\la,\,\la,\,7\la}$ & 9
& \branch{\brep24+1}{\brep33}{7+1+1} & \PH{11}, \nocomposition
& \PR{57}, $t\in\QQ(\sqrt{-7})$ 
\\ 
& $\led{\fr12,\,\la,\,2\la,\,6\la}$ & 
& \branch{\brep24+1}{\brep33}{6+2+1} & \PH{12},     \comp{\GT3}{3}
& \PR{17}, $t=-3$ 
\\ 
& $\led{\fr12,\,\la,\,3\la,\,5\la}$ & 
& \branch{\brep24+1}{\brep33}{5+3+1} & \PH{13}, \nocomposition
& \PR{33} $t=128/3$ 
\\ 
& $\led{\fr12,\,2\la,\,3\la,\,4\la}$ & 
& \branch{\brep24+1}{\brep33}{4+3+2} & \PH{14},   \comp{\GT3}{3}
& \PR{18}, $t=-3$ 
\\ 
& $\led{\fr13,\,\la,\,\la,\,8\la}$ & 10 &
\branch{\brep25}{\brep33+1}{8+1+1} & \PH{7}, \nocomposition
& \PR{49}, $t\in\QQ(\sqrt{-2})$ 
\\ 
& $\led{\fr13,\,\la,\,2\la,\,7\la}$ & 
& \branch{\brep25}{\brep33+1}{7+2+1} & \PH{8},   \nocomposition
& \PR{35}, $t=81/32$ 
\\ 
& $\led{\fr13,\,\la,\,4\la,\,5\la}$ & 
& \branch{\brep25}{\brep33+1}{5+4+1} & \PH{9},   \nocomposition
& \PR{28}, $t=-80$ 
\\ 
& $\led{\fr13,\,2\la,\,3\la,\,5\la}$ & 
& \branch{\brep25}{\brep33+1}{5+3+2} & \PH{10},    \nocomposition
& \PR{32}, $t=32/5$ 
\\ 
& $\led{\la,\,\la,\,\la,\,9\la}$ & 12 &
\branch{\brep26}{\brep34}{9+1+1+1} & \PH{1}, \comp{\GT4}{\CT3}
& \PR{55}, $j=0$ 
\\ 
& $\led{\la,\,\la,\,2\la,\,8\la}$ & 
& \branch{\brep26}{\brep34}{8+2+1+1} & \PH{2}, \compp{\GT3}{\GT2}{2}
& \PR{14}, $j=1728$ 
\\ 
& $\led{\la,\,2\la,\,3\la,\,6\la}$ & 
& \branch{\brep26}{\brep34}{6+3+2+1} & \PH{3}, \comp{\GT4}{3}, \comp{\GT3}{4}
& \PR{24}, $t=-8$ 
\\ 
& $\led{\la,\,\la,\,5\la,\,5\la}$ & 
& \branch{\brep26}{\brep34}{5+5+1+1} & \PH{4}, \comp{6}{\HT2}
& \PR{43}, $t\in\QQ(\sqrt{5})$ 
\\ 
& $\led{2\la,\,2\la,\,4\la,\,4\la}$ & 
& \branch{\brep26}{\brep34}{4+4+2+2} & \PH{5}, \compp{\GT2}{\GT{\CT3\!}}{2}, \comp{\GT3}{\DDT}
& \PR{12}, $j=1728$ 
\\ 
& $\led{3\la,\,3\la,\,3\la,\,3\la}$ & 
& \branch{\brep26}{\brep34}{3+3+3+3} & \PH{6}, \comp{\GT4}{\CT3\!}, \compp{\CT3}{\HT2\!}{\HT2\!}
& \PR{53}, $j=0$ 
\\ \hline
\end{tabular}
\caption{Gauss-to-Heun  transformations with two continuous
parameters, or from $\hpgde{\fr12,\fr13,\,\la}$.} \label{clasfig}
\end{center}
\end{table}

\begin{table} \small
\begin{center}
\begin{tabular}{@{}llclll@{}} 
\hline \multicolumn{2}{c}{Local exponent differences} & 
$\degr$ & Branching pattern 
& The covering,
& Characterization  
\\ \cline{1-2}
hypergeom. & Heun's & 
& above singular points & its composition & of \S \ref{heunclass}
\\  \hline
$\led{\fr12,\fr14,\,\la}$ & $\led{\fr12,\fr12,\,\la,\,3\la}$ & 4
& \branch{\brep21+1+1}{\brep41}{3+1} & \PH{36}, \nocomposition
& \PR{48}, $t\in\QQ(\sqrt{-2})$ 
\\ 
& $\led{\fr12,\fr12,\,2\la,\,2\la}$ & 
& \branch{\brep21+1+1}{\brep41}{2+2} 
& \PH{35}, \comp{\GT2}{\HT2} & \PR{37}, $t\in\QQ(\sqrt{2})$
\\ 
& $\led{\fr12,\fr14,\,\la,\,4\la}$ & 5
& \branch{\brep22+1}{\brep41+1}{4+1} & \PH{44}, \nocomposition
& \PR{60}, $j\in\QQ(i)$ 
\\ 
& $\led{\fr12,\fr14,\,2\la,\,3\la}$ & 
& \branch{\brep22+1}{\brep41+1}{3+2} & \PH{29},  \nocomposition
& \PR{27}, $t=-80$ 
\\ 
& $\led{\fr12,\,\la,\,2\la,\,3\la}$ & 6
& \branch{\brep23}{\brep41+2}{3+2+1} & \PH{25}, \comp{\GT2}{3}
& \PR{16}, $t=-3$ 
\\ 
& $\led{\fr14,\fr14,\,\la,\,5\la}$ & 
& \branch{\brep23}{\brep41+1+1}{5+1} & \PH{42}, \nocomposition
& \PR{44}, $t\in\QQ(i)$ 
\\ 
& $\led{\fr14,\fr14,\,3\la,\,3\la}$ & 
& \branch{\brep23}{\brep41+1+1}{3+3} & \PH{43},   \comp{3}{\HT2}
& \PR{39}, $t\in\QQ(\sqrt{3})$ 
\\ 
& $\led{\la,\,\la,\,2\la,\,4\la}$ & 8 &
\branch{\brep24}{\brep42}{4+2+1+1} & \PH{40}, \compp{\GT2}{\GT2}{2}
& \PR{13}, $j=1728$  
\\ 
& $\led{\la,\,\la,\,3\la,\,3\la}$ & 
& \branch{\brep24}{\brep42}{3+3+1+1} & \PH{20},   \comp{\GT2}{4}, \comp{\AT4}{\HT2}
& \PR{23}, $t=-8$ 
\\ 
& $\led{2\la,\,2\la,\,2\la,\,2\la}$ & 
& \branch{\brep24}{\brep42}{2+2+2+2} & \PH{41}, $2\!\times\!2\!\times\!2$
& \PR{10}, $j=1728$ 
\\ 
$\led{\fr12,\fr15,\,\la}$ & $\led{\fr12,\,\la,\,\la,\,3\la}$ & 5
& \branch{\brep22+1}{\brep51}{3+1+1} & \PH{37}, \nocomposition
& \PR{58}, $t\in\QQ(\!\sqrt{-15})$ 
\\ 
& $\led{\fr12,\,\la,\,2\la,\,2\la}$ & 
& \branch{\brep22+1}{\brep51}{2+2+1} & \PH{45},   \nocomposition
&  \PR{42}, $t\in\QQ(\sqrt{5})$ 
\\ 
& $\led{\fr15,\,\la,\,\la,\,4\la}$ & 6
& \branch{\brep23}{\brep51+1}{4+1+1} & \PH{42}, \nocomposition
& \PR{46}, $t\in\QQ(i)$ 
\\ 
& $\led{\fr15,\,\la,\,2\la,\,3\la}$ & 
& \branch{\brep23}{\brep51+1}{3+2+1} & \PH{24},    \nocomposition
& \PR{34}, $t=128/3$ 
\\ 
$\led{\fr12,\fr16,\,\la}$ & $\led{\la,\,\la,\,\la,\,3\la}$ & 6
& \branch{\brep23}{\brep61}{3+1+1+1} & \PH{38}, \comp{\GT2}{\CT3}
& \PR{54}, $j=0$ 
\\ 
& $\led{\la,\,\la,\,2\la,\,2\la}$ & 
& \branch{\brep23}{\brep61}{2+2+1+1} &  \PH{39}, \comp{\GT2}{3}, \comp{3}{\HT2}
& \PR{22}, $t=-8$ 
\\  
$\led{\fr13,\fr13,\,\la}$ & $\led{\fr13,\fr13,\,\la,\,3\la}$& 4
& \branch{\brep31+1}{\brep31+1}{3+1} &  \PH{46},   \nocomposition
& \PR{5}, $j=1728$ 
\\ 
& $\led{\fr13,\fr13,\,2\la,\,2\la}$ & 
& \branch{\brep31+1}{\brep31+1}{2+2} &  \PH{47},  \nocomposition
& \PR{40}, $t\in\QQ(\sqrt{3})$ 
\\ 
& $\led{\la,\,\la,\,2\la,\,2\la}$ & 6
& \branch{\brep32}{\brep32}{2+2+1+1} &\PH{28}, \comp{\GT{\CT3}}{2}
& \PR{11}, $j=1728$ 
\\ $\led{\fr13,\fr14,\,\la}$     
& $\led{\fr13,\,\la,\,\la,\,2\la}$ & 4
& \branch{\brep31+1}{\brep41}{2+1+1} & \PH{36}, \nocomposition
& \PR{50}, $t\in\QQ(\sqrt{-2})$ 
\\ $\led{\fr14,\fr14,\,\la}$     
& $\led{\la,\,\la,\,\la,\,\la}$ & 4
& \branch{\brep41}{\brep41}{1+1+1+1} & \PH{48}, \comp{2}{\HT2}
& \PR{9}, $j=1728$ 
\\ \hline
\end{tabular}
\caption{Other hypergeometric-to-Heun transformations.}
\label{clasfig2} 
\end{center}
\end{table}

Tables \ref{clasfig} and \ref{clasfig2} are renditions of \cite[Tables 1, 2, 3]{HeunClass}.
The first two columns give the 
exponent differences (up to the sign) of the hypergeometric and Heun equation 
under a pull-back transformation. Let $E$ be the hypergeometric equation. 
The third column gives the degree $D$ of the transformation.
The fourth column gives the branching pattern of the pull-back covering.
The branching pattern is given by 3 partitions of $D$ separated by the equality sign.
The partitions specify the branching orders of the covering in the 3 fibers above
the singular points of $E$. 
The notation \brep{k}{n} means the sum of $n$ repeated $k$'s in a partition.
It represents $n$ points with the branching order $k$ above a singularity of $E$ with the 
exponent difference restricted to $1/k$; those $n$ points would be non-singular
with an appropriate gauge prefactor $\theta(x)$ in (\ref{eq:algtransf}).
The number of bracketed branching orders is equal to the number of restricted 
exponent differences. The number of non-bracketed branching orders is equal to 4; 
they represent the 4 singular points
of the pulled-back Heun equation. The total number of points in the three fibers
is  equal to $\degr+2$, as required for the Belyi coverings $\PP^1\to\PP^1$ by the Hurwitz formula;
see \cite[Lemma 3.2]{HeunClass}.

The fifth column identifies the pull-back coverings. The $H_k$ notation refers to the list of
48 Belyi coverings in \cite[Table 4]{HeunClass} not normalized yet by a M\"obius transformation.
Most of the coverings can be found in the explicit formulas of \S \ref{sec:hpghe} here,
as arguments of the hypergeometric functions.
The coverings \PH{1} to \PH{38} appear in Herfurtner's list \cite{Herfurtner91}
of elliptic surfaces over $\PP^1$ with 4 singular fibers; the ${\cal J}(X,Y)$-expressions
in \cite[Table 3]{Herfurtner91} are projectivized Belyi coverings and give
the $j$-invariants of the elliptic surfaces up to the multiple 1728.
The numbering \PH{1} to \PH{38} agrees with \cite[Table 1]{Reiter},
where Movasati and Reiter observe that 38 of Herfurtner's 50 cases of elliptic surfaces
give rise to pull-backs 
from $\hpgde{1/2,1/3,\la}$ to Heun equations. In Table \ref{clasfig} here, 
the coverings \PH{1} to \PH{30} and \PH{36}, \PH{37}, \PH{38}
appear in the pull-backs 
specifically from $\hpgde{1/2,1/3,\la}$, while the coverings \PH{31} to \PH{35} 
appear in pull-back transformations 
with 2 or  3 parameters. The coverings \PH{39} to \PH{48} appear in pull-backs 
to Heun functions from the hypergeometric equations
different from $\hpgde{1/2,1/3,\la}$. The coverings \PH{20}, \PH{24}, \PH{25}, \PH{28}, \PH{29},
\PH{34}, \PH{35}, \PH{37}, \PH{38}, \PH{42}, \PH{47} appear twice in Tables \ref{clasfig} and
\ref{clasfig2}, while \PH{36} three times, as their branching patterns can be parsed for the
Heun-to-Gauss reductions in multiple ways.

The fifth column also tells which coverings are compositions
of lower degree coverings, and indicates the compositions by the component degrees.
The notation reveals a few more specifics about the compositions;
see the beginning of Appendix B for details.

The last column of Tables \ref{clasfig} and \ref{clasfig2} exhibits the P-numbers of the
Heun-to-Gauss reductions assigned by a new perspective, described in the next section.
Relatedly, the last column adds minimal information about the $t$-values of the pulled-back
Heun equations. The $t$-values are cross-ratios of the 4 singular points of Heun's equations
(or of pulled-back Fuchsian equations with 4 singularities, even if the location of 3
singularities is not normalized to $x=0$, $x=1$, $x=\infty$). A permutation of the 4 singular points
generally produces an orbit of six $t$-values:
\begin{equation} \label{eq:allts}
t, \; 1-t, \; \frac1t, \; \frac1{1-t}, \; \frac{t}{t-1}, \; 1-\frac1t.
\end{equation}
As it is well-known, the set of six values can be represented by one number, the $j$-invariant:
\begin{equation} \label{eq:invj0}
j(t)=\frac{256\,(t^2-t+1)^3}{t^2\,(t-1)^2}.
\end{equation}
The $\cal J$-invariant used in \cite{Herfurtner91} is the Belyi map ${\cal J}(t)=j(t)/1728$.
Its version appears in hypergeometric transformation (\ref{eq:trd6}) below.
The last column of Tables \ref{clasfig}, \ref{clasfig2}  additionally indicates:
\begin{itemize}
\item the most frequent $j$-values 1728 and 0, if $t\in\{-1,2,1/2,-\omega,1+\omega\}$;
\item or a representative $t$-value, if it is in $\QQ$ and $j(t)\neq 1728$.
\item or the number field for the $t$-values, if $t\not\in\QQ$ and $j(t)\in\QQ\setminus\{0\}$;
\item or the number field for the $j$-value, if $j(t)\not\in\QQ$.
\end{itemize}
The last case appears only twice, with the coverings \PH{21} and \PH{44}. These two coverings
are not defined over $\QQ$ either, but over $\QQ(\omega)$ and $\QQ(i)$, respectively.
Technically speaking, the notations \PH{21} and \PH{44} represent pairs of coverings
related by the complex conjugation. Therefore the strict count of involved Belyi coverings
(or of their {\em dessin d'enfant}) is 50 rather than 48.
For some technical purposes, the corresponding Heun-to-Gauss
reductions P61 and P60 can be counted as pairs of different transformations as well.

\subsection{Classification by Heun equations}
\label{heunclass}

The main application of the list of possible Heun-to-Gauss reductions is,
of course, finding out whether an encountered Heun function or equation is reducible to
a hypergeometric one. Tables \ref{clasfig}, \ref{clasfig2} are not convenient for looking up
a match with the parameters of an encountered Heun equation, as even the tuples of four 
exponent differences are listed disorderly. Modifications by the fractional-linear transformations
of  Heun equations should also be recognized, hence additional (to $j$) invariants
of the fractional-linear action are helpful. The fractional-linear transformations of both
hypergeometric and Heun functions are recalled in Appendix B. The additional invariants
are derived and listed in Appendix C.

Of all Heun's parameters $t,q,a,b,c,d$, the most characteristic one is $t$, which is the location of the fourth
singularity. Therefore $t$ or its $j$-invariant are the most sensible main criteria
for sorting Heun-to-Gauss reductions. To formulate the full ordering uniquely,
we adopt the criteria in \cite{VidunasHoeij} for sorting a more complicated set of
non-parametric Heun-to-Gauss reductions. Thereby the presentation of results
here and in \cite{VidunasHoeij} are consistent with each other.  Here is the simplified set
(up to accounting for free parameters) of the sorting criteria in \cite{VidunasHoeij}.
They are enough to determine a unique ordering on the set of 61 transformations,
and the resulting list is usable for matching and reducing an encountered Heun's equation.

The settled sorting criteria are the following:
\begin{itemize}
\item[\refpart{a}] the first criterium is the $j$-invariant;
\item[\refpart{b}] the second criterium is the local exponent differences of the Heun equation;
\item[\refpart{c}] the last criterium is the degree of the covering.
\end{itemize}
In a similar hierarchical manner, the $j$-invariants are sorted by the following criteria:
\begin{itemize}
\item[\refpart{a1}] the number field where the $j$-invariant is defined;
\item[\refpart{a2}] the number field where the $t$-values are defined;
\item[\refpart{a3}] the leading coefficient of the minimal polynomial in $\ZZ[x]$
for the $j$-invariant. 
\end{itemize}
Note that for $j\in\QQ$ the number in \refpart{a3} is the denominator of $j$.
The number fields (either for the $j$-invariant or the $t$-values) are ordered by the following  criteria:
\begin{itemize}
\item[\refpart{f1}] the field degree, hence $\QQ$ precedes quadratic extensions;
\item[\refpart{f2}] quadratic extensions $\QQ(\sqrt{a})$ are ordered as follows:
\begin{itemize}
\item[\refpart{f1a}] real quadratic fields (with $a>0$) precede imaginary quadratic fields (with $a<0$);
\item[\refpart{f1b}] the fields with the same sign of $a$ are ordered
by the increasing $|a|$.
\end{itemize}
\end{itemize}
The positive integers in \refpart{a3} 
are ordered as follows:
\begin{itemize}
\item[\refpart{i1}] the product of the primes dividing the integer;
\item[\refpart{i2}] by the increasing value.
\end{itemize}
Except for the \refpart{i1}-part and for using the absolute value in \refpart{f1b},
all other numeric specifics are ordered in the increasing order.
The sets of local  exponent differences are ordered as follows:
\begin{itemize}
\item[\refpart{b1}] in each tuple the four 
exponent differences are ordered
by {\em putting the free parameters at the end}, and 
the numeric values of the (positive, rational) restricted 
exponent differences are ordered firstly their denominators,
then secondly by the numerators.
\item[\refpart{b2}] the tuples are {\em first compared by the number of restricted 
exponent differences (hence the tuples with more free parameters have precedence)}, then lexicographically, from their first elements, and the elements are matched first by their denominators then by the numerators.
\end{itemize}
These criteria break all ties in the list of 61 transformations, as mentioned.
In particular, no sorting criteria is necessary for the accessory parameters $q$
(or their invariants). The highlighted text in \refpart{b1}--\refpart{b2} accounts for the
presence of free parameters, absent in the criteria for the non-parametric list in \cite{VidunasHoeij}.

\begin{table} \small 
\begin{tabular}{llllcll}  
\hline Id & $j$-invariant
& Exponent &   \multicolumn2l{Covering} & $\hpgo21$ equation & Other trans- \\
\cline{4-5} & & differences & Id & $\degr$ & & formations \\
\hline
\PR{1} & $2^63^3=1728$ & $\led{\la,\,\la,\,\lb,\,\lc}$ & \PH{32} & 2 & $\led{\la,\,\lb/2,\,\lc/2}$ & --- \\
\PR{2} & & $\led{\la,\,\la,\,\lb,\,\lb}$ & \PH{31} & 4 & $\led{1/2,\la/2,\lb/2}$ & \PR{1}, $2_H$ \\
\PR{3} & & $\led{\la,\,\la,\,2\la,\,\lb}$ & \PH{35} & 4 & $\led{1/2,\,\la,\,\lb/4}$ & \PR{1}  \\
\PR{4} & & $\led{1/2,1/2,\la,2\la}$ & \PH{27} & 6 & $\led{1/2,1/3,\la/2}$ & \PR{1}; $2^H$ \\
\PR{5} & & $\led{1/3,1/3,\la,3\la}$ & \PH{46} & 4 & $\led{1/3,1/3,\la}$ & ---; \PR{1}, \PR{6} \\
\PR{6} & & & \PH{19} & 8 & $\led{1/2,1/3,\la/2}$ & \PR{1}, \PR{5} \\
\PR{7} & & $\led{2/3,\la,\la,6\la}$  & \PH{15} & 8 & $\led{1/2,1/3,\la}$ & \PR{1} \\
\PR{8} & & $\led{2/3,2\la,3\la,3\la}$  & \PH{17} & 8 & $\led{1/2,1/3,\la}$ & \PR{1} \\
\PR{9} & & $\led{\la,\,\la,\,\la,\,\la}$  & \PH{48} & 4 & $\led{1/4,1/4,\la}$ &
$2_H$; \PR{1}, \PR{2}, \PR{10}, $4_H$ \\
\PR{10} & & & \PH{41}  & 8 & $\led{1/2,1/4,\la/2}$ & \PR{1}, \PR{2}, \PR{9}, $2_H$, $4_H$ \\
\PR{11} & & $\led{\la,\,\la,\,2\la,\,2\la}$ & \PH{28} & 6 & $\led{1/3,1/3,\la}$ &
\PR{1}; \PR{2}, \PR{3}, \PR{12}, $2_H$ \\
\PR{12} & & & \PH{5} & 12 & $\led{1/2,1/3,\la/2}$ & \PR{1}, \PR{2}, \PR{3}, \PR{11}, $2_H$ \\
\PR{13} & & $\led{\la,\,\la,\,2\la,\,4\la}$ & \PH{40} & 8 & $\led{1/2,1/4,\la}$ & \PR{1}, \PR{3} \\
\PR{14} & & $\led{\la,\,\la,\,2\la,\,8\la}$ & \PH{2} & 12 & $\led{1/2,1/3,\la}$ & \PR{1}, \PR{3} \\
\PR{15} & $2^413^3/3^2$ & $\led{1/2,\la,2\la,\lb}$ & \PH{34} & 3 & $\led{1/2,\,\la,\,\lb/3}$ & ---  \\
\PR{16} & & $\led{1/2,\la,2\la,3\la}$ & \PH{25} & 6 & $\led{1/2,1/4,\la}$ & \PR{15} \\ %
\PR{17} & & $\led{1/2,\la,2\la,6\la}$ & \PH{12} & 9 & $\led{1/2,1/3,\la}$ & \PR{15} \\
\PR{18} & & $\led{1/2,2\la,3\la,4\la}$ & \PH{14} & 9 & $\led{1/2,1/3,\la}$ & \PR{15} \\
\PR{19} & $2^273^3/3^4$ & $\led{\la,\,2\la,\,\lb,\,2\lb}$ & \PH{34} & $3$ & $\led{1/3,\,\la,\,\lb}$ & --- \\ %
\PR{20} & & $\led{\la,\,3\la,\,\lb,\,3\lb}$ & \PH{47} & 4 & $\led{1/2,\,\la,\,\lb}$ & --- \\ %
\PR{21} & & $\led{1/3,2/3,\la,2\la}$ & \PH{25} & 6 & $\led{1/2,1/3,\la/2}$ & \PR{19} \\
\PR{22} & & $\led{\la,\,\la,\,2\la,\,2\la}$ & \PH{39} & 6 & $\led{1/2,1/6,\la}$ & \PR{19}, $2_H$ \\ %
\PR{23} & & $\led{\la,\,\la,\,3\la,\,3\la} $  & \PH{20} & 8 & $\led{1/2,1/4,\la}$ & \PR{20}, $2_H$ \\
\PR{24} & & $\led{\la,\,2\la,\,3\la,\,6\la} $ & \PH{3} & 12 & $\led{1/2,1/3,\la}$ & \PR{19}, \PR{20} \\
\hline
\end{tabular}  \centering
\caption{Parametric Heun-to-Gauss reductions equations with a $t$ value in $\{-1,-3,-8\}$
 up to fractional-linear transformations.} \label{heunred0}
\end{table}

\begin{table} \small 
\begin{tabular}{llllcll} 
\hline Id & $j$-invariant
& Exponent &   \multicolumn2l{Covering} & $\hpgo21$ equation & Other trans- \\
\cline{4-5 }& \hfill \scriptsize (number field for $t$) & differences & Id & $\degr$ & & formations \\
\hline
\PR{25} & $2^43^37^3/5^2$ & $\led{1/2,2/3,2\la,3\la}$ & \PH{30} & 5 & $\led{1/2,1/3,\la}$ & --- \\
\PR{26} & $13^337^3/3^45^4$ & $\led{1/3,2/3,\la,5\la}$ & \PH{24} & 6 & $\led{1/2,1/3,\la}$ & --- \\
\PR{27} & $6481^3/3^85^2$ & $\led{1/2,1/4,2\la,3\la}$ & \PH{29} & 5 & $\led{1/2,1/4,\la}$ & --- \\
\PR{28} & & $\led{1/3,\la,4\la,5\la}$ & \PH{9} & 10 & $\led{1/2,1/3,\la}$ & --- \\
\PR{29} & $2^67^397^3\!/3^65^4\!$ & $\led{2/3,\la,2\la,5\la}$ & \PH{16} & 8 & $\led{1/2,1/3,\la}$ & --- \\
\PR{30} & $2^4757^3\!/3^67^2$ & $\led{1/2,1/3,3\la,4\la}$ & \PH{23} & 7 & $\led{1/2,1/3,\la}$ & --- \\
\PR{31} & $7^3127^3\!/2^23^65^2$ & $\led{1/2,2/3,\la,4\la}$ & \PH{29} & 5 & $\led{1/2,1/3,\la}$ & --- \\
\PR{32} & & $\led{1/3,2\la,3\la,5\la}$ & \PH{10} & 10 & $\led{1/2,1/3,\la}$ & --- \\
\PR{33} & $7^32287^3\!/2^63^25^6$ & $\led{1/2,\la,3\la,5\la}$ & \PH{13} & 9 & $\led{1/2,1/3,\la}$ & --- \\
\PR{34} & & $\led{1/5,\la,2\la,3\la}$ & \PH{24} & 6 & $\led{1/2,1/5,\la}$ & --- \\
\PR{35} & $4993^3/2^23^87^4$ & $\led{1/3,\la,2\la,7\la}$ & \PH{8} & 10 & $\led{1/2,1/3,\la}$ & --- \\
\PR{36} & $19^31459^3\!/2^43^65^67^2\!$ & $\led{1/2,1/3,2\la,5\la}$ & \PH{22} & 7 &
$\led{1/2,1/3,\la}$ & ---  \\
\PR{37} & $2^33^311^3$ \hfill \scriptsize ($\sqrt{2}$) & $\led{1/2,1/2,\la,\la}$ & \PH{35}
 & 4 & $\led{1/2,1/4,\la/2}$ & $2_H$; $2^H$ \\
\PR{38} & $2^43^35^3$ \hfill \scriptsize ($\sqrt{3}$) & $\led{1/2,1/2,\la,\la}$ & \PH{28}
 & 6 & $\led{1/2,1/3,\la/3}$ & $2_H$; $2^H$ \\
\PR{39} & $2^2193^3\!/3$ \hfill \scriptsize ($\sqrt{3}$) & $\led{1/4,1/4,\la,\la}$ & \PH{43}
 & 6 & $\led{1/2,1/4,\la/3}$ & $2_H$ \\
\PR{40} & $2^753^3\!/3^3$ \hfill \scriptsize ($\sqrt{3}$) & $\led{1/3,1/3,\la,\la}$ & \PH{47}
 & 4 & $\led{1/3,1/3,\la/2}$ &  ---; \PR{41}, $2_H$ \\
\PR{41} & & & \PH{20} & 8 & $\led{1/2,1/3,\la/4}$ & \PR{47}, $2_H$; \PR{40} \\
\PR{42} & $2^417^3$ \hfill \scriptsize ($\sqrt{5}$) & $\led{1/2,\la,2\la,2\la}$ & \PH{45}
 & 5 & $\led{1/2,1/5,\la}$ & --- \\
\PR{43} & $2^{14}31^3\!/5^3$ \hfill \scriptsize ($\sqrt{5}$) & $\led{\la,\,\la,\,5\la,\,5\la}$ & \PH{4}
 & 12 & $\led{1/2,1/3,\la}$ & $2_H$ \\
\PR{44} & $2^23^313^3\!/5^4$ \hfill \scriptsize ($\sqrt{-1}$) & $\led{1/4,1/4,\la,5\la}$ & \PH{42}
 & 6 & $\led{1/2,1/4,\la}$ & --- \\
\PR{45} &$-2^4109^3\!/5^6$\hfill \scriptsize ($\sqrt{-1}$)  & $\led{1/2,1/2,\la,5\la}$ & \PH{26}
 & 6 & $\led{1/2,1/3,\la}$ & ---; $2^H$ \\
\PR{46} & & $\led{1/5,\la,\la,4\la}$ & \PH{42} & 6 & $\led{1/2,1/5,\la}$ & --- \\
\PR{47} & $-2^519^3\!/3^6$ \hfill \scriptsize ($\sqrt{-2}$) & $\led{1/2,1/2,1/3,\la}$ & \PH{36}
 & 4 & $\led{1/2,1/3,\la/4}$ & ---; $2^H$ \\
\PR{48} & $2\!\cdot\! 47^3\!/3^8$ \hfill \scriptsize ($\sqrt{-2}$) & $\led{1/2,1/2,\la,3\la}$ & \PH{36}
 & 4 & $\led{1/2,1/4,\la}$ & ---;  $2^H$ \\
\PR{49} & & $\led{1/3,\la,\la,8\la}$ & \PH{7} & 10 & $\led{1/2,1/3,\la}$ & --- \\
\PR{50} & $-2^6239^3\!/3^{10}\!$ \hfill \scriptsize ($\sqrt{-2}$) & $\led{1/3,\la,\la,2\la}$ & \PH{36}
 & 4 & $\led{1/3,1/4,\la}$ & --- \\
\PR{51} & 0 \hfill \scriptsize ($\sqrt{-3}$) & $\led{\la,\,\la,\,\la,\,\lb}$ & \PH{33}
 & 3 & $\led{1/3,\,\la,\,\lb/3}$ & ---  \\
\PR{52} & & $\led{1/3,1/3,1/3,\la}$ & \PH{38} & 6 & $\led{1/2,1/3,\la/6}$ & \PR{51} \\
\PR{53} & & $\led{\la,\,\la,\,\la,\,\la}$ & \PH{6} & 12 & $\led{1/2,1/3,\la/3}$ & \PR{51} \\
\PR{54} & & $\led{\la,\,\la,\,\la,\,3\la}$ & \PH{38} & 6 & $\led{1/2,1/6,\la}$ & \PR{51} \\
\PR{55} & & $\led{\la,\,\la,\,\la,\,9\la}$ & \PH{1} & 12 & $\led{1/2,1/3,\la}$ & \PR{51} \\
\PR{56} & $\!-2^{11}11^3\!/3^37^4\;$ \hfill \scriptsize ($\sqrt{-3}$) & $\led{1/3,1/3,\la,7\la}$ & \PH{18}
& 8 & $\led{1/2,1/3,\la}$ & --- \\
\PR{57} & $5^343^3\!/2^67^3$ \hfill \scriptsize ($\sqrt{-7}$) & $\led{1/2,\la,\la,7\la}$ & \PH{11}
& 9 & $\led{1/2,1/3,\la}$ & --- \\
\PR{58} & $\!-269^3\!/2^{10}3^5$ \hfill \scriptsize ($\sqrt{-15}$) & $\led{1/2,\la,\la,3\la}$ & \PH{37}
& 5 & $\led{1/2,1/5,\la}$ & --- \\
\PR{59} & $71^3\!/2^43^{3}5$ \hfill \scriptsize ($\sqrt{-15}$) & $\led{1/2,1/3,1/3,\la}$ & \PH{37}
& 5 & $\led{1/2,1/3,\la/5}$ & --- \\
\PR{60} & $\frac{(1+i)^{12}(3-2i)^3}{(2-i)^2}$ & $\led{1/2,1/4,\la,4\la}$ & \PH{44}
& 5 & $\led{1/2,1/4,\la}$ & --- \\
\PR{61} & $\!-\omega2^4\frac{(1-2\omega)^3(7+6\omega)^3}{(1+2\omega)^{6}(3+2\omega)^{2}}\!$ & $\led{1/2,1/3,\la,6\la}$ & \PH{21} & 7 & $\led{1/2,1/3,\la}$ & --- \\ \hline
\end{tabular}  \centering
\caption{Other parametric Heun-to-hypergeometric reductions.}  \label{heunred}
\end{table}

The resulting ordering is displayed in Tables \ref{heunred0} and \ref{heunred}.
It starts with a list of 14 transformations to Heun equations with $t\in\{-1,2,1/2\}$.
A detailed identification shows that the Heun equations for these transformations are 
the same as for the well-known
quadratic transformation  $\hpgde{\la,\lb,\lc}\stackrel{2}{\longleftarrow}\heunde{\la,\,\la,2\lb,\,2\lc}$ up to the parameter identification and the fractional-linear symmetries of Appendix A.
We mark this quadratic transformation by P1.

The last column of Tables \ref{heunred0} and \ref{heunred}
displays other considered transformations possible for the same Heun equation as for the currently
numbered one. A semicolon there separates the possible transformations
that are composition factors of the currently numbered, from the other possible transformations
(after a semicolon, if present). Accordingly, P1 is listed in the last column for the P2--P14 entries,
but it is not a composition factor for P5 and P9. The notation $2_H$, $2^H$, $4_H$ refers to
Heun-to-Heun transformations. It is explained in Appendix B, and the transformations are
considered in \S \ref{Heun2Heun}.

The other cases of different transformations with (generically) the same Heun equation
are within the sequences P15--P18, P19--P24, P40--P41 and P51--P55. In particular,
the reducible Heun equations with $j(t)=0$ can be obtained by the cubic transformation P51.
Likewise, all transformations to Heun equations with $t\in\{-3,4,-1/3,4/3,1/4,3/4\}$
are specializations of the two-parametric P15.
However, there are two unrelated transformations P19, P20 giving Heun equations
with $t\in\{-8,9,-1/8,9/8,1/9,8/9\}$.

A more detailed identification of transformed Heun equations is possible  by considering
the formulas of \S \ref{sec:hpghe}, or additional invariants of the fractional-linear transformations
described in Appendix C. The following theorem can be considered as the main result of this paper.

\begin{theorem}\label{th:t}
Suppose that Heun's equation $(\ref{Heun})$ is (a specialization of a) parametric
pull-back transformation of a hypergeometric equation, and the monodromy group
of the hypergeometric equation is not cyclic or dihedral.
Then the $j$-invariant $(\ref{eq:invj0})$ and the $4$ local exponent differences of
Heun's equation gives one of the following situations:
\begin{itemize}
\item[\refpart{i}] $j(t)=1728$, and at least $2$ 
exponent differences are equal up to multiplication by $-1$;
\item[\refpart{ii}] $j(t)=0$, and at least $3$ 
exponent differences are equal up to multiplication by $-1$;
\item[\refpart{iii}] $j(t)=35152/9$, 
and the Heun equation is $\heunde{\pm1/2,\la,\pm2\la,\lb}$
for some $\la,\lb\in\CC$; 
\item[\refpart{iv}] $j(t)=1556068/81$,  
and the Heun equation is $\heunde{\la,\pm2\la,\lb,\pm2\lb}$
or $\heunde{\la,\pm3\la,\lb,\pm3\lb}$ for some $\la,\lb\in\CC$;
\item[\refpart{v}] the $j$-invariant is listed in the second column of Table $\ref{heunred}$
among the entries \PR{25}--\PR{40}, \PR{42}--\PR{50}, \mbox{\PR{56}--\PR{59}}, and the 
exponent differences satisfy the respective pattern in the third column up to multiplication by $-1$.
\item[\refpart{vi}] up to the conjugation $i\mapsto -i$, $\omega\mapsto-\omega-1$,
the $j$-invariant is listed in the \PR{60} or \PR{61} entry of Table $\ref{heunred}$, and the 
exponent differences satisfy the respective pattern in the third column
up to multiplication by $-1$.
\end{itemize}
\end{theorem}
\begin{proof} A detailed inspection of Tables  \ref{heunred0} and \ref{heunred},
and additional analysis of Heun equations with the same $j$-invariant and matching pattern of 
exponent differences proves the statement.
\end{proof}

Theorem \ref{th:t} gives necessary conditions for a given Heun equation to be reducible
to a hypergeometric one by the considered pull-back transformations.
For a set of sufficient conditions, see Theorem \ref{th:tinv}.

The number of different $j$-invariants in the reducible Heun equations is 32,
counting pairs of conjugate values of P60 and P61 as two different numbers.
The number of different Heun equations up to M\"obius transformations is 38.

\begin{table} \small
\renewcommand{\arraystretch}{1.25}
\begin{tabular}{lll} 
\hline Id 
& $t-$values  & $a+b=c$ \\   \hline
\PR{1}/\PR{14} 
& $-1$, $2$, $\frac{1}{2}$ & $1+1=2$ \\
\PR{15}/\PR{18} 
& $-3,\,4,\,-\frac{1}{3},\,\frac{4}{3},\,\frac{1}{4},\,\frac{3}{4}$ & $1+3=2^2$ \\
\PR{19}/\PR{24} 
& $-8,\,9,\,-\frac{1}{8},\,\frac{9}{8},\,\frac{1}{9},\,\frac{8}{9}$ & $1+2^3=3^2$ \\
\PR{25} 
& $-4,\,5,\,-\frac{1}{4},\,\frac{5}{4},\,\frac{1}{5},\,\frac{4}{5}$ & $1+2^2=5$\\
\PR{26} 
& $-\frac{16}{9},\,\frac{25}{9},\,-\frac{9}{16},\,\frac{25}{16},\,\frac{9}{25},\,\frac{16}{25}$ & $3^2+4^2=5^2$ \\
\PR{27}/\PR{28} 
& $-80,\,81,\,-\frac{1}{80},\,\frac{81}{80},\,\frac{1}{81},\,\frac{80}{81}$ & $1+2^4\cdot 5=3^4$ \\
\PR{29} 
& $-\frac{25}{2},\,\frac{27}{2},\,-\frac{2}{25},\,\frac{27}{25},\,\frac{2}{27},\,\frac{25}{27}$ & $2+5^2=3^3$ \\
\PR{30} 
& $-27,\,28,\,-\frac{1}{27},\,\frac{28}{27},\,\frac{1}{28},\,\frac{27}{28}$ & $1+3^3=2^2\cdot 7$ \\
\PR{31}/\PR{32} 
& $-\frac{27}{5},\,\frac{32}{5},\,-\frac{5}{27},\,\frac{32}{27},\,\frac{5}{32},\,\frac{27}{32}$ & $5+3^3=2^5$ \\
\PR{33}/\PR{34} 
& $-\frac{125}{3},\,\frac{128}{3},\,-\frac{3}{125},\,\frac{128}{125},\,\frac{3}{128},\,\frac{125}{128}$
& $3+5^3 =2^7$ \\
\PR{35} 
& $-\frac{49}{32},\,\frac{81}{32},\,-\frac{32}{49},\,\frac{81}{49},\,\frac{32}{81},\,\frac{49}{81}$
& $2^5+7^2=3^4$ \\
\PR{36} 
& $-\frac{125}{64},\,\frac{189}{64},\,-\frac{64}{125},\,\frac{189}{125},\,\frac{64}{189},\,\frac{125}{189}$
& $2^6 +5^3=3^3\cdot7$ \\
\PR{37} 
& $-16\pm 12 \sqrt{2},\,17\pm12\sqrt{2},\,\frac{1}{2}\pm\frac{3\sqrt{2}}{8}$
& $(1-\sqrt{2})^2+(\sqrt{2})^5 =(1+\sqrt{2})^2$  \\
\PR{38} 
&  $-7\pm 4\sqrt{3}$, $8\pm 4\sqrt{3}$, $\frac{1}{2}\pm \frac{\sqrt{3}}{4}$
& $1+(2+\sqrt{3})^2=(2-\sqrt{3})(1+\sqrt{3})^4$ \\
\PR{39} 
& $-96\pm 56\sqrt{3},\,97\pm 56\sqrt{3},\,\frac{1}{2}\pm\frac{7\sqrt{3}}{24}$
& $1+\sqrt{3}\,(2-\sqrt{3})(1+\sqrt{3})^6=(2+\sqrt{3})^4$  \\
\PR{40}/\PR{41} 
& $-26\pm 15\sqrt{3},\,27\pm15\sqrt{3},\,\frac{1}{2}\pm\frac{5\sqrt{3}}{18}$
& $(2+\sqrt{3})^2+(2-\sqrt{3})=(\sqrt{3})^3(1+\sqrt{3})$  \\
\PR{42} 
& $-8\pm4\sqrt{5}, 9\pm4\sqrt{5}, \frac12\pm\frac{\sqrt{5}}{4}$
& $\left(\frac{1+\sqrt{5}}{2}\right)^3+\left(\frac{1-\sqrt{5}}{2}\right)^3=2^2$ \\
\PR{43} 
& $\frac{-123\pm55\sqrt{5}}{2},\,\frac{125\pm 55\sqrt{5}}{2},\,\frac{1}{2}\pm\frac{11\sqrt{5}}{50}$
& $\left(\frac{1+\sqrt{5}}{2}\right)^5+\left(\frac{\sqrt{5}-1}{2}\right)^5=(\sqrt{5})^3$ \\
\PR{44} 
& $\frac{-7\pm 24i}{25},\,\frac{32\pm24i}{25},\,\frac{1}{2} \pm\frac{3i}{8}$
& $(2+i)^2+(1+i)^6=(2-i)^2$ \\
\PR{45}/\PR{46} 
& $\frac{8\pm 44i }{125},\, \frac{117\pm44i}{125},\,\frac{1}{2}\pm\frac{11i}{4}$
& $(2+i)^3+(1+i)^4=-(2-i)^3 $ \\
\PR{47} 
& $\frac{4\pm 10\sqrt{-2}}{27},\,\frac{23\pm10\sqrt{-2}}{27},\,\frac{1}{2}\pm \frac{5\sqrt{-2}}{4}$
& $(1+\sqrt{-2})^3+(\sqrt{-2})^3=(1-\sqrt{-2})^3$  \\
\PR{48}/\PR{49} 
& $\frac{17\pm 56\sqrt{-2}}{81},\,\frac{64\pm 56\sqrt{-2}}{81},\,\frac{1}{2}\pm\frac{7\sqrt{-2}}{16}$
& $(1-\sqrt{-2})^4+(\sqrt{-2})^7=(1+\sqrt{-2})^4$ \\
\PR{50} 
& $\frac{2\pm 22\sqrt{-2} }{243},\,\frac{241\pm 22\sqrt{-2}}{243},\,\frac{1}{2}\pm \frac{11\sqrt{-2}}{2}$
& $(1+\sqrt{-2})^5+(1-\sqrt{-2})^5=-(\sqrt{-2})^2$  \\
\PR{51}/\PR{55} 
& $\frac{1}{2}\pm\frac{\sqrt{-3}}{2}$ &  $(-\omega)+(1+\omega)=1$  \\
\PR{56} 
& $\frac{27\pm39\sqrt{-3}}{98},\,\frac{71\pm 39\sqrt{-3}}{98},\,\frac{1}{2}\pm\frac{13\sqrt{-3}}{18}$
& $(3+2\omega)^2+(1+\omega)(1+2\omega)^3=\omega(1-2\omega)^2$  \\
\PR{57} 
& $\frac{-87\pm 91\sqrt{-7} }{256},\,\frac{343\pm91\sqrt{-7}}{256},\,\frac{1}{2}\pm\frac{13\sqrt{-7}}{98}$
& $\left(\frac{1+\sqrt{-7}}{2}\right)^7+(\sqrt{-7})^3=\left(\frac{1-\sqrt{-7}}{2}\right)^7$  \\
\PR{58} 
& $\frac{243\pm171\sqrt{-15}}{1024}\!,
\frac{781\pm171\sqrt{-15}}{1024}\!, \frac{1}{2}\!\pm\!\frac{19\sqrt{-15}}{54}\!$
& $\!\left(\frac{\!1+\sqrt{-15\!}}{2}\right)^{\!4} \! \left(\frac{\!3-\sqrt{-15\!}}{2}\right) \! + \! 3^3 \! = \!
\left(\frac{\!1-\sqrt{-15\!}}{2}\right)^{\!4} \!\! \left(\frac{\!-3-\sqrt{-15\!}}{2}\right)\!\!$ \\ %
\PR{59} 
& $\frac{-7\pm33\sqrt{-15}}{128},\frac{135\pm33\sqrt{-15}}{128},\frac{1}{2}\pm\frac{11\sqrt{-15}}{90}$
& $\!\left(\frac{1+\sqrt{-15}}{2}\right)^{\!3}\!+3\sqrt{-15}=\left(\frac{1-\sqrt{-15}}{2}\right)^{\!3}$  \\
\PR{60} 
& $-2i,\,1+2i,\,1-\frac{i}{2},\,\frac{i}{2},\,\frac{1-2i}{5},\,\frac{4+2i}{5}$ & $1+(1+i)^2=(1+2i)$  \\ %
\PR{61} 
& $\!\frac{1+3\omega}{4}\!,\frac{3-3\omega}{4}\!,\frac{-8-12\omega}{7}\!,\frac{15+12\omega}{7}\!,\frac{1-4\omega}{9}\!,\frac{8+4\omega}{9}\!$ & $(1+3\omega)+(1-\omega)^2(2+\omega)=2^2$ \\ %
\hline
\end{tabular} \centering
\caption{The $t$-values of reducible Heun equations.} \label{tb:abc}
\end{table}

\subsection{Arithmetic observation of the $t$-values}
\label{tvalues}

The second column of Table \ref{tb:abc} gives all possible values of the $t$-parameter
of the reducible Heun equations considered here. A look at the rational $t$-values in the
P19--P36 cases reveals several nicely factorizable integers like 81, 32, 125, 128 in the numerators
or denominators of the $t$-values. Algebraic $t$-values have nice factorization expressions as well.
For example, in cases P43, P50, P57 we have
\[
\frac{-123+55\sqrt{5}}2=-\left(\frac{1-\sqrt{5}}2\right)^{\!10},\quad
\frac{241+22\sqrt{-2}}{243}=-\frac{\left(1+\sqrt{-2}\right)^{10}}{3^5},\quad
\frac{-87+91\sqrt{-7}}2=\left(\frac{1-\sqrt{-7}}2\right)^{\!14}.
\]
Often all six $t$-values in an orbit under fractional-linear
transformations factorize rather remarkably. 

A compact expression for the classical orbit (\ref{eq:allts}) of six $t$-values
is an identity
\begin{equation} \label{eq:abc}
a+b=c,
\end{equation}
where the vector $(a,b,c)$ of numbers is a multiple of $(t,1-t,1)$.
The orbit of six $t$-values is recovered as $\{a/c,b/c,c/a,c/b,-a/b,-b/a\}$.
When $t$ is a rational number, a convenient $(a,b,c)$ triple is obtained
by clearing the denominators of $(t,1-t,1)$, so that $a,b,c$ in (\ref{eq:abc})
are pairwise co-prime integers. For example, the $abc$ identity for P26 is
$9+16=25$, reminding the most famous Pythagorean triangle.
For $t$-values in an algebraic number field $K$,
its proportional identities $a+b=c$ 
in algebraic integers of $K$  can be considered
as a single point $(a:b:c)$ on the projective line $\PP^1$ over $K$.
If the ring of algebraic integers is a {\em principal ideal domain},
the numbers $a,b,c$ can be chosen to be ``co-prime"
but the identity can be multiplied by units.

The third column of Table \ref{tb:abc} spells out arithmetic $abc$-identities
defining the $t$-values of the encountered Heun equations, including those over
algebraic number fields. 
Many of the identities are indeed attractive, as only factors of small norm
are involved. This is not accidental. Our coverings are Belyi maps,
and it is known \cite{Beckmann89} that Belyi maps degenerate only modulo
primes (or prime ideals) of small size. Arithmetic properties of the cross-ratio $t$
of presumably different 4 points clearly reflect the primes of bad reduction for the Belyi maps.
This is the reason why the numbers in the listed $a+b=c$ identities are highly factorizable,
or only a few primes are involved.

Equations like (\ref{eq:abc}) with $a,b,c$ prescribed to involve only a small set $S$ of primes
are known in number theory as {\em S-unit equations}. They typically have only finitely many
solutions up to scalar multiplication \cite{Schlickewei}. Diophantine equations for
``highly factorizable" integers are enjoy wide popular interest. In particular,
solving Fermat's equation $x^n+y^n=z^n$ in integers was a famous open problem for centuries.
After Wiles' resolution of Fermat's problem in 1995, a prominent generalizing
arithmetic conjecture is the {\em abc-conjecture} of Masser and Osterl\'e \cite{Granville2002}.
It states that for any real $\varepsilon>0$ there should be finitely many
identities (\ref{eq:abc}) with co-prime integers $a,b,c$ such that
the {\em quality ratio}
\begin{equation} \label{eq:abcq}
Q(a,b,c):=\frac{\log \max(|a|,|b|,|c|) }{\log \mbox{rad}(abc)}
\end{equation}
is greater than $1+\varepsilon$. Here the {\em radical} rad$(n)$ is the product of
prime numbers dividing $n$. For example, the quality ratio of $3+125=128$,
which gives the $t$-values for P33/P34, is equal to $\log(128)/\log(30)\approx 1.426565$.
Currently there are over 200 examples  known \cite{desmitabc} with the quality ratio greater than 1.4.

For comparison, the $abc$-theorem \cite[Proposition 2]{Granville2002} for polynomials states that
for any identity (\ref{eq:abc}) with co-prime polynomials $a,b,c\in\CC[x]$ of maximal degree $D$,
the number of different roots of the product $abc$ is at least $D+1$.
This is a familiar consequence of the Hurwitz formula, as in \cite[Lemma 3.2]{HeunClass}.
The bound $D+1$ is attained when the rational function $a/c$ is a Belyi map and
its value at $x=\infty$ is $0,1$ or $\infty$ on $\PP^1$.

There is a generalization of the $abc$-conjecture over number fields \cite{abcnf},
where the definition of the quality ratio in (\ref{eq:abcq}) is adjusted as follows.
The numerator is replaced by the logarithm of the {\em height} of $(a:b:c)\in\PP^1(K)$,
and rad$(abc)$ is replaced by the product of (the absolute value of) the field discriminant
and the norms of the {\em prime ideals} (or {\em places, non-archimedean norms})
which reduce $(a:b:c)$ to a trivial point like $(1:0:1)$.
For example, the quality ratio for the P48/P49 identity is computed as
$\log\max(3^4,2^7,3^4)/\log(8\cdot3\cdot2\cdot3)\approx1.074487$,
while the quality ratio for the P43 identity is equal to $\log\max(1^5,1^5,5^3)/\log(5\cdot5)=1.5$.
Among the encountered number fields, only $\QQ(\sqrt{-15})$ is not a principal ideal domain.
This field defines two transformations with the same covering $H_{37}$. 
The quality ratio for the P58 identity is equal to
$\log\max(2^9,2^9,3^5)/\log(15\cdot2\cdot2\cdot3)\approx1.201305$,
while for the P59 identity it is equal to
$\log\max(4^3,9\cdot15,4^3)/\log(15\cdot2\cdot3\cdot5\cdot2)\approx0.721110$.

In total, Table \ref{tb:abc} contains 12 identities $a+b=c$ of the quality ratio greater than 1.
Two identities (for P33/P34 and P43) have the quality ratio greater than 1.4.
The encountered $t$-values are of relatively small size, and the $abc$ identities are not
groundbreaking. However, the identity for P45/P46 recently brought a \$50 prize
to Fred W.~Helenius \cite{Gaussians}.
Considering Belyi maps and cross ratios of 4 points in the three branching fibers
may be a fruitful strategy for finding interesting $abc$ triples, especially
over algebraic number fields. The non-parametric ``hyperbolic" hypergeometric-to-Heun
transformations \cite{VidunasHoeij} give more known high quality examples, such as
\[ 1+2^55^23=7^4, \quad
\left(\frac{1+\sqrt{-7}}{2}\right)^{\!13}\!+\sqrt{-7}=\left(\frac{1-\sqrt{-7}}{2}\right)^{\!13}\!, \!\qquad
\left(\frac{\sqrt{5}-1}{2}\right)^{\!12}\!+2^43^2\sqrt{5}=\left(\frac{1+\sqrt{5}}{2}\right)^{\!12}\!
\]
with the respective quality ratios 1.455673, 1.707222, 1.697794,
and a new identity in $\QQ(\sqrt{-14})$ with the quality ratio 1.581910.

\section{Deriving Heun-to-hypergeometric identities}
\label{heunhpgid}

With the list of suitable Belyi coverings at hand, pull-back transformations (\ref{eq:algtransf})
between hypergeometric and Heun equations are obtained by normalizing the Belyi maps
with M\"obius transformations (so that the singularities of the pulled-back equation would
indeed be at $x=\infty$, $x=0$, $x=1$, $x=t$), and by choosing suitable gauge prefactors $\theta(x)$.
The parameters $a,b,c,d$ and $A,B,C$ of the related differential equations (\ref{HGE})
and (\ref{Heun}) 
are determined by the 
exponent differences assigned to the singular points.
The accessory parameter $q$ can be determined by Lemma \ref{th:acq} here below,
or by considering the first power series terms of a two-term Heun-to-Gauss identity.
A pull-back transformation can be composed with the fractional-linear symmetries of
the hypergeometric and Heun equations, described   in Appendix A.

The role of the prefactor $\theta(x)$ is to get rid of {\em irrelevant singularities}
and shift a local exponent at each $x=0$, $x=1$, $x=t$ to the value 0,
as prescribed by a Riemann scheme in (\ref{eq:P}).
The direct pull-back transformation with $\theta(x)=1$ would typically give a Fuchsian equation
with several non-logarithmic singular points where the 
exponent difference is equal to 1; we call them {\em irrelevant singularities}. 
They can be turned into non-singular points by shifting their 
exponents to the values 0 and 1.
The possible irrelevant singular points are above $z=\infty$, and above
the finite singular points ($z=0$, $z=1$) where the restricted 
exponents of the hypergeometric equation are $0,-1/k$ rather than $0,1/k$. 
Besides, the relevant singular points above $z=\infty$
would typically have only non-zero local exponents, and a local exponent for all of them
except $x=\infty$ has to be shifted to the value $0$. The prefactor will have the form
$\theta(x)=\prod_i (x-\sigma_i)^{-\xi_i}$, where $\sigma_i$ are all the $x$-points where
the local exponents need to be shifted, and $\xi_i$ is the local exponent at $\sigma_i$
to be shifted to $0$. The local exponents at $x=\infty$ then shift by the sum of all $\xi_i$'s.
It is convenient to use the $P$-notation of the Riemann scheme,
as demonstrated in appendix formula (\ref{eq:ptransform}).

The prefactor is not needed if there is only one point $x=\infty$ above \mbox{$z=\infty$}.
Then one can choose the local exponents $0,1/k$ (rather than $0,-1/k$) at the restricted points
$z\in\{0,1\}$ to have no irrelevant singularities under the direct pullback.
The rational function $\varphi(x)$ defining the covering is then a polynomial.
Maier \cite{Maier} classified all parametric transformations between
hypergeometric and Heun equations without a prefactor.
A list of seven transformations was obtained.
Using our identification, these are  five Maier's indecomposable
\begin{eqnarray} \label{eq:maier}
\PR1
: & \hpgde{\alpha,\beta,\gamma}\pback{2}\heunde{\alpha,\alpha,2\beta,2\gamma}, \nonumber\\ 
\PR{15} 
: & \hpgde{\fr12,\alpha,\beta}\pback{3}\heunde{\fr12,\alpha,2\alpha,3\beta}, \nonumber\\ 
\PR{47} 
: & \hpgde{\fr12,\fr13,\alpha}\pback{4}\heunde{\fr12,\fr12,\fr13,4\alpha},\\ 
\PR{51} 
: & \hpgde{\fr13,\alpha,\beta}\pback{3}\heunde{\alpha,\alpha,\alpha,3\beta}, \nonumber\\ 
\PR{59} 
: & \hpgde{\fr12,\fr13,\alpha}\pback{5}\heunde{\fr12,\fr13,\fr13,5\alpha},\nonumber 
\end{eqnarray}
and two composite transformations:
\begin{eqnarray} 
\PR3 
: & \hpgde{\fr12,\alpha,\beta}\pback{2}\hpgde{\alpha,\alpha,2\beta}\pback{2}
\heunde{\alpha,\alpha,2\alpha,4\beta}, \\ 
\PR{52} 
: & \hpgde{\fr12,\fr13,\alpha}\pback{2}\hpgde{\fr13,\fr13,2\alpha}
\pback{3}\heunde{\fr13,\fr13,\fr13,6\alpha}. \nonumber 
\end{eqnarray}
The coverings are \PH{32} to \PH{38} in a mixed up order.
A proper normalization by the fractional-linear symmetries of Appendix A is required to avoid the prefactor.
In addition, several more transformations without a prefactor are given in \cite{Maier}  for
the degenerate Heun equation with $ab=q=0$. The function $\varphi(x)$ does not   have
to be a polynomial then, as the points above $z=\infty$ immediately have a local exponent $0$.

Two-term identities between the Heun and hypergeometric functions are derived
by identifying standard local solutions at the corresponding points of the related
Heun and hypergeometric 
equations. By fractional-linear transformations,
any singular $x$-point can be chosen as $x=0$ and its projection as $z=0$.
Then we are identifying the standard Heun and hypergeometric series at $x=0$.
This determines a two-term identity up to fractional-linear transformations
(\ref{frlin2f1}) and (\ref{frlinhb})--(\ref{frlinha}). The prefactor $\theta(x)$ has to
be normalized to the value $\theta(x)=1$ at $x=0$. If the 
exponent difference at $z=0$ is an unrestricted parameter, 
changing its sign gives essentially the same two-term identity.
More generally, the following choices of $x=0$ give the same two-term identities up to
the fractional-linear transformations and change of parameters:
\begin{itemize}
\item[\refpart{i}] the $x$-points 
with the same branching index and above the same point of $\PP_z^1$;
\item[\refpart{ii}]  the $x$-points 
with the same branching index, if they are in different fibers with the same branching pattern, and
either the local exponents at the corresponding $z$-points are the same, or the 
exponent differences at both $z$-points are free parameters.
\end{itemize}

A pull-back transformation (\ref{eq:algtransf}) between the hypergeometric and (or) Heun equations
might fail to produce two-term identities between the hypergeometric and (or) Heun solutions
only if  all singular points of the transformed equation lie above non-singular points of
the starting equation. The singularities of the transformed equation are then apparent,
and the pull-back covering is typically not a Belyi map so that 
\cite[Proposition 3.3]{HeunClass}  likely applies.
An example of such a transformation in given in \cite[Remark 5.9]{VidunasDihTr};
it is a composition of $\hpgde{1/2,1/2,1/2}\stackrel{4}{\longleftarrow}\heunde{1,1,1}$
and $\hpgde{1,1,1}\stackrel{3}{\longleftarrow}\heunde{3,2,2}$ with general ramification fibers
in the second transformation. On the other hand, transformation identities between the hypergeometric
and Heun functions might formally exist without a pull-back 
between their equations. For example, the linear function $\hpg21{\!-1,\,b}{c}{z}$ can be formally
transformed to any (hypergeometric or Heun) polynomial. Part 2 of \cite[Lemma 2.1]{VidunasFE}
indicates that this situation can occur only if we start with a hypergeometric function actually 
satisfying a first order Fuchsian equation.

If the exponent difference at $z=0$ is restricted to $1/k$ with $k\in\ZZ$,
the choices $0,1/k$ and $0,-1/k$ of local exponents give different identities.
Changing the sign of the 
exponent difference at $z=0$ basically gives
an identity between the other two local solutions (with non-zero local exponents)
at $x=0$ and $z=0$. For transformations between hypergeometric functions,
this situation is captured by \cite[Lemma 2.3]{VidunasFE}.
Here is a reformulation for identities between the Heun and hypergeometric functions.

\begin{lemma} \label{localx0}
Suppose that we have the identity $(\ref{eq:H-2F1 gen})$ coming from a pull-back transformation
between the corresponding hypergeometric and Heun equations.
Then $\varphi(x)^{1-C}\sim Kx^{1-c}$ as $x\rightarrow 0$ for
some constant $K$, and the following  identity holds:
$$
\heun{t}{q_1}{{1+a-c,\,}{1+b-c}}{{2-c;\,}{d}}{x}=\Theta(x)\,\hpg{2}{1}{1+A-C,\,1+B-C\,}{2-C}{\,\varphi(x)},
$$
where $q_1=q-(c-1)(a+b-c-d+d\, t+1)$ and
$\Theta(x)=\theta(x)\,\varphi(x)^{1-C}\big/K x^{1-c}.$
\end{lemma}
\begin{proof} The lemma is proved by a straightforward identification of the other canonical local solutions
of both equations at $x=0$ and $z=0$.
\end{proof}

The accessory parameter $q$ of pulled-back 
Heun's equation can be determined later on 
by considering power series expansions at $x=0$ in a supposed two-term identity and comparing the first couple of terms in the power series. The value of $q$ is given by the following lemma.  
\begin{lemma} \label{th:acq}
Suppose that we have the  identity $(\ref{eq:H-2F1 gen})$ coming from a pull-back transformation
between the corresponding hypergeometric and Heun equations with
$$ 
\varphi(x)=\lambda\,x+O(x^2), \qquad \theta(x)=1+\mu\,x+O(x^2)
$$ 
as $x\rightarrow 0$. Then
$
\displaystyle q= c\,t \left( \mu + \frac{A\,B\,\lambda}C \right).
$
\end{lemma}
\begin{proof}
Expanding both sides of (\ref{eq:H-2F1 gen}) in the power series at $x=0$ gives
$$
1+\frac{q}{c\,t}\,x+O(x^2) = 1+\mu\,x+\frac{A\,B}{C}\,\lambda\,x+O(x^2).
$$
\end{proof}
Note particularly, that if the covering $\varphi(x)$ branches at $x=0$
and the prefactor $\theta(x)$ is absent, then $q=0$ (because $\lambda=\mu=0$);
check formulas (\ref{H-2F1 1tr}) and (\ref{eq:c4}) below.

\section{Identities between the Heun and hypergeometric functions}
\label{sec:hpghe}

In this section we  briefly survey pull-back transformations between the Heun or hypergeometric functions,
and then we present parametric Gauss-to-Heun transformations.
Only parametric pull-backs 
from the hypergeometric equations with cyclic or dihedral
monodromy are not considered here (see \cite{VidunasHDD}).

We remark that we do not consider identities like
\begin{equation} \label{eq:joyce}
\heun{4}{1/2}{1/2,1/2}{1;1/2}{-\frac{4s(s-1)(s+2)(s+1)}{(2s+1)^2}}
=\sqrt{1+2s} \; \hpg21{1/2,1/2}{1}{\frac{s^3(s+2)}{2s+1}},
\end{equation}
with rational functions in both arguments, or algebraic radicals in an argument,
even if they contain a free parameter. 
Formula (\ref{eq:joyce}) is a reparametrized version
(without argument radicals) of Joyce's identity \cite{Joyce94}, cited in \cite[(24)]{Valent07} as well.
We  consider neither relations  of Heun's equations with an apparent singularity to $\hpgo32$
and other generalized  hypergeometric functions, nor relations to sums of contiguous $\hpgo21$ functions,
illustrated in \cite{Letesier94} and \cite[\S 5]{MaierPH}.

\subsection{Transformations between hypergeometric functions}

Pull-back transformations between the hypergeometric equations  give
algebraic transformations between hypergeometric functions of the form
\begin{equation}\label{HGE to HGE}
 \hpg{2}{1}{a,\,b\,}{c}{\,x} =\theta(x)\,
\hpg{2}{1}{A,\,B}{C}{\,\varphi(x)}.
\end{equation}
The classical transformations were obtained by Gauss, Goursat,
Riemann, Kummer. Here is an example of a cubic transformation
with one free parameter $a$:
\begin{equation}
\hpg{2}{1}{3a,\,\frac{1}{3}-a\,}{2a+\frac{5}{6}}{\,x}
=\left(1-4x\right)^{-3a}\hpg{2}{1}{a,\,a+\frac{1}{3}}{2a+\frac{5}{6}}{\,\frac{27x}{(4x-1)^3}}.
\end{equation}
This is the transformation
$\hpgde{1/2,1/3,\la}\pback{3}\hpgde{1/2,\,\la,\,2\la}$, with $\la=1/6-2a$.
The classical transformations have degree at most 6,
namely $\hpgde{1/2,1/3,\la}\pback{6}\hpgde{\la,\,\la,\,4\la}$
and $\hpgde{1/2,1/3,\la}\pback{6}\hpgde{2\la,\,2\la,\,2\la}$. The latter formula is given by 
\begin{equation} \label{eq:trd6}
\hpg{2}{1}{6a,\,2a+\frac{1}{3}}{4a+\frac{2}{3}}{\,x}=(1-x+x^2)^{-3a}
\;\hpg{2}{1}{a,\,a+\frac13}
{2a+\frac{5}{6}}{\frac{27}{4}\frac{x^2\,(x\!-\!1)^2}{(x^2\!-\!x\!+\!1)^3}}.
\end{equation}

Pull-back transformations between the hypergeometric equations, and
subsequently, algebraic transformations of the Gauss
hypergeometric functions, are systematically 
classified\footnote{Transformation (\ref{eq:trd6}) is presented in \cite[(28)]{VidunasFE}
with a misprint in the lower parameter $2a+5/6$. Here is a list of other inaccuracies in \cite{VidunasFE}:
(ii) the case $a\neq 0$ in (29) should be multiplied by $-1$; (iii) proof of Theorem 6.1 should refer to
\cite[Theorem 5.1]{VidunasDihTr}; (iv) uniqueness claims on pg.~162 and Remark 7.1 are dubious, 
especially with $\ell/k\in\ZZ$; see \cite[(5.47)]{VidunasDihTr} and \cite[\S 5.4]{HeunClass}. 
Furthermore, the question of Remark 7.1 about existence of Gauss-to-Gauss pull-backs 
that do not yield two-term hypergeometric formulas is answered in \cite[Remark 5.7]{VidunasDihTr}
positively with the example $\hpgde{1/2,1/2,1/2}\pback{12}\hpgde{3,2,2}$, as mentioned here
in \S \ref{heunhpgid} right affer the listing \refpart{i}--\refpart{ii}. }
in \cite{VidunasFE}. 

The well-known quadratic transformations of Gauss hypergeometric functions
have two free {parameters}:
\begin{eqnarray} \label{eq:quadr1}
\hpg{2}{1}{2a,\,2b\,}{a+b+\frac{1}{2}}{\,x} \equal
\hpg{2}{1}{a,\,b}{a+b+\frac{1}{2}}{\,4x(1-x)},\\ \label{eq:quadr2}
\hpg{2}{1}{2a,\,a-b+\frac12}{a+b+\frac{1}{2}}{\,x\,}\! \equal
(1-x)^{-2a}\hpg{2}{1}{a,\,b\,}{a+b+\frac{1}{2}}{\,-\frac{4x}{(x-1)^2}},\\
\label{eq:quadr3}
 \hpg{2}{1}{\,2a,\;b\,}{2b}{\,x} \equal
\left(1-\frac{x}{2}\right)^{\!-2a} \hpg{2}{1}{a,\,a+\frac{1}{2}\,}
{b+\frac{1}{2}}{\frac{x^2}{(2-x)^2}}.
\end{eqnarray}
The first two formulas are related by fractional-linear transformations (\ref{frlin2f1}),
whereas (\ref{eq:quadr3}) is not equivalent 
up to the fractional-linear transformations (on either $\PP_x^1$ or $\PP_z^1$),
as noted by Askey \cite{RA} and Maier \cite[Remark 4.1.2]{MaierPH}.
The dividing difference is the choice of the point $x=0$: it is a non-branching point
in (\ref{eq:quadr1})--(\ref{eq:quadr2}) but a branching point in the last formula.

But formula (\ref{eq:quadr3}) can be derived from (\ref{eq:quadr1}) by the following argument.
The functions
$$ x^{-2a}\,\hpg21{2a,\,a-b+\frac12}{2a-2b+1}{\frac1x},\qquad
(1-2x)^{-2a}\,\hpg21{a,\,a+\frac12}{1+a-b}{\frac1{(1-2x)^2}}
$$
are among the 24 Kummer solutions of the differential equations for 
the left-hand side and the right-hand side of (\ref{eq:quadr1}) respectively. Therefore the two functions
satisfy the same Fuchsian equation of order 2. We multiply both functions by $x^{2a}$,
make the substitutions $x\mapsto 1/x$ and $b\mapsto a-b+1/2$ and obtain the two functions
in (\ref{eq:quadr3}) up to a constant multiple on the right-hand side.
Those two functions satisfy the same Fuchsian equation of order 2, have the same value 
and the same local exponent at a regular singular point (with a non-integer 
exponent difference in general), so they must be equal in a neighborhood of $x=0$,
and (\ref{eq:quadr3}) follows.

An example 
of a non-classical Gauss-to-Gauss transformation 
is $\hpgde{1/2,1/3,1/7}\pback{10}\hpgde{1/3,1/7,2/7}$ given by 
\begin{eqnarray} \label{eq:hpghpg}
\hpg{2}{1}{\frac5{42},\,\frac{19}{42}}{\frac57}
{x} \equal {\textstyle
\left(1-\frac{19}9x-\frac{343}{243}x^2+\frac{16807}{6561}x^3\right)}^{-1/28}\times\nonumber\\
&& \hpg{2}{1}{\frac1{84},\,\frac{29}{84}}{\frac67}
{\frac{x^2\,(1-x)\,(49x-81)^7}{4\,(16807x^3-9261x^2-13851x+6561)^3}}.
\end{eqnarray}
The degree 10 rational function is 
one of our Belyi coverings $H_8$ up to the M\"obius transformations. This is not surprising, 
as specialization of the 
exponent difference to $\la=1/7$ turns the Heun equation for P32 to
$E(1/3,1/7,2/7)$. In the same way, all 61 Heun-to-Gauss 
parametric transformations can be specialized to Gauss-to-Gauss transformations
classified in \cite{VidunasFE}.

The identities like (\ref{eq:hpghpg}) can be verified by checking the power series at $x=0$.
But the common region of convergence usually appears to be small.
For example, (\ref{eq:hpghpg})  does not hold at $x=1$ or $x=81/49$ for the standard
analytic branches of $\hpgo21$ functions, as can be checked numerically.

\subsection{Quadratic hypergeometric-to-Heun transformations $\mbox{(P1)}$}
\label{HeHquadratic}

Quadratic Gauss-to-Heun transformations apply to Gauss hypergeometric
functions without any restriction of their parameters. The underlying reason is that
a quadratic covering branches only above 2 points, and if the branching is above the
singularities of the hypergeometric equation, there are exactly 4 points above those singularities.
Here are explicit formulas:
\begin{eqnarray} \label{H-2F1 1tr}
\heun{-1}{0}{2a,\,2b}{2c-1;\,a+b-c+1}{\,x} \equal
\hpg{2}{1}{a,\,b\,}{c}{\,x^2},\\ \label{H-2F1 2tr}
\heun{2}{4ab}{2a,\,2b}{c;\,2a+2b-2c+1}{\,x} \equal
\hpg{2}{1}{a,\,b\,}{c}{\,x(2-x)},\\  \label{H-2F1 3tr}
\heun{\frac12}{\,2a b}{\,2a,\,2b}{c;\;c}{\,x} \equal
\hpg{2}{1}{a,\,b\,}{c}{\,4x(1-x)}.
\end{eqnarray}
They were first indicated by Kuiken in \cite{Kuiken}.
Other possible polynomials $\varphi(x)$ for quadratic transformations
between the hypergeometric and Heun equations  are $1-x^2$, $(1-x)^2$, $(2x-1)^2$.
Fractional-linear transformations of $P$-symbols for the 192 Heun functions
and the related Kummer's 24 hypergeometric functions give a set of another 30
rational functions of degree 2 that transform the general hypergeometric equation
to Heun's equations (with a prefactor, in general). The 30 rational functions are given
in \cite{Kuiken} in the context of the degenerate case $ab=q=0$.

Like for hypergeometric quadratic transformations (\ref{eq:quadr1})--(\ref{eq:quadr3}),
we have two different choi\-ces for $x=0$: a branching point and a non-branching point.
Accordingly, identities (\ref{H-2F1 2tr}) and (\ref{H-2F1 3tr}) are related
by fractional-linear transformations  (\ref{frlin2f1}), (\ref{frlinha})--(\ref{frlinhb}),
whereas identity (\ref{H-2F1 1tr}) cannot be related to them by the
fractional-linear transformations. To derive (\ref{H-2F1 1tr}) from (\ref{H-2F1 2tr}),
one can observe that the functions
$$\heun{-1}{0}{2a,\;2b}{2a+2b-2c+1;\,c}{\,1-x},\qquad
\hpg{2}{1}{a,\,b\,}{a+b-c+1}{\,(1-x)^2 }$$
satisfy the same Heun equation as both sides of (\ref{H-2F1 2tr}),
and have the same local exponent and value at $x=1$.
Therefore they must be generally equal; formula (\ref{H-2F1 1tr})
is then obtained after the substitution $x\mapsto 1-x$, $c\mapsto a+b-c+1$.


\subsection{Heun-to-Heun transformations} \label{Heun2Heun}

Existence of quadratic and quartic
transformations was pointed out by Erd\'elyi in \cite[Vol. 3]{BE}.
Examples of these transformations are given by Maier in
\cite[\S 4]{MaierPH}. Here are two alternative formulas of the quadratic
transformations:
\begin{eqnarray} \label{heunquadr}
\heun{s^2}{q_1}{{2a,\,}{2a-b+1}}{{b;\,}{2a-b+1}}{\,x} \!\!\equal\!\!
\left(1+\frac{x}{s}\right)^{\!-2a}
\heun{\!\frac{4s}{(1+s)^2}}{q}{a,\,a+\frac{1}{2}}{b;\,\frac{1}{2}}{\frac{4xs}{(x+s)^2}},\\
\label{heunquadr2}
\heun{\frac{s^2}{2s-1}}{\!\frac{2abs+4qs(s-1)}{2s-1}\!}{2a,\,b}{b;\;b}{x} \!\!\equal\!\!
\left(1-\frac{x}{s}\right)^{\!-2a}
\heun{\!\frac{1}{4s(1-s)}}{q}{a,\,a+\frac{1}{2}}{b;2a-b+1}{\frac{x\,(x-1)}{(x-s)^2}},\qquad
\end{eqnarray}
where $q_1=(1+s)^2q-2 a b s$. The two formulas
are related by a series of fractional-linear transformations (and reparametrizations).
The transformation of local exponents is given by
\begin{eqnarray}
\label{eq:quadtr3} && \heunde{1/2,1/2,\,\la,\,\lb}\pback{\HT2}\heunde{\la,\,\la,\,\lb,\,\lb}.
\end{eqnarray}
All choices for $x=0$ give two-term identities related by fractional-linear transformations
(\ref{frlinhb})--(\ref{frlinha}). 

A quartic transformation can be obtained by composing two versions
of the quadratic transformation:
\begin{eqnarray} \label{eq:quadtr4}
\heunde{\fr12,\fr12,\fr12,\,\la}\pback{\HT2}
\heunde{\fr12,\fr12,\,\la,\,\la}\pback{\HT2}
\heunde{\la,\,\la,\,\la,\,\la}.
\end{eqnarray}
In the composition, we restrict particularly
$s\mapsto 1/2s$, $b\mapsto 2a+1/2$ and $a\mapsto 2a$, $b\mapsto 2a+1/2$ in the two versions.
Remarkably, transformation of the parameters $t$ and $q$ simplifies greatly.
After setting $t=s^2/(2s-1)$ in the composition, we recognize the transformation
\begin{equation} \label{eq:quarticl}
\heun{t}{4q}{4a,\,2a+\frac12}{2a+\frac12;2a+\frac12}{x}=
\left(1-\frac{x^2}{t}\right)^{\!-2a}
\heun{t}{q}{a,\,a+\frac{1}{2}}{2a+\frac12;\,\frac{1}{2}}{\frac{4tx(x-1)(x-t)}{(x^2-t)^2}}
\end{equation}
as in \cite[Theorem 4.2]{MaierPH}.
%
The composite degree 4 covering 
happens to be the Belyi covering $\PH{31}$.
Up to the M\"obius transformation $z\to1/z$, the starting Heun equation
for the quartic transformation is a general Lam\'e equation $\heunde{\fr12,\fr12,\fr12,\,\la}.$

Other Heun-to-Heun transformations are possible only for the very special case
 of Lam\'e equation $\heunde{1/2,1/2,1/2,1/2}$.  This can be seen by considering necessary
branching patterns. Other pull-back coverings cannot not be Belyi maps
(as we wish only 4 singular points), hence they ramify above all 4 singular $z$-points.
The four fibers would contain at least $2D+2$ different points, and we want at least $2D-2$
of them to be non-singular after a pull-back. But each fiber has at most $\lfloor D/2 \rfloor$
non-singular points, quickly leading to $\heunde{1/2,1/2,1/2,1/2}$.
As recalled in \cite[\S 3]{Valent07}, Carlitz \cite{Carlitz60} solved this equation
by giving an explicit basis of solutions. For the Heun equation in the canonical form,
the two independent solutions of Carlitz are
\begin{equation}
y_\pm(x)=\exp\left(\pm\sqrt{q}\int_0^x\frac{du}{\sqrt{u(u-1)(u-t)}}\right).
\end{equation}
This is an integral of a holomorphic differential on the general {\em Legendre elliptic curve}
\begin{equation} \label{legendre}
w^2=u(u-1)(u-t).
\end{equation}
Any isogeny between Legendre elliptic curves transforms the holomorphic
differentials to each other up to a scalar multiple, since the space of holomorphic differentials
on elliptic curves is one-dimensional. Vice versa, the particular branching pattern of
the coverings $\varphi(x)$ ensures the transformations $u\mapsto\varphi(x)$ of holomorphic
differentials. It follows that any degree transformations of $\heunde{1/2,1/2,1/2,1/2}$ exist,
and they correspond to the isogenies of Legendre elliptic curves. In particular,
here are the cubic and and an alternative quartic transformations:
\begin{eqnarray}
\heun{   \frac{s^3\left( s-2 \right)}{1- 2\,s} }{q\,{\left( 1 -
2s \right)}^2}{\,0,\,\frac{1}{2}\,}{\frac{1}{2};\,\frac{1}{2}}{x} \equal
\heun{\frac{s\,(s-2)^3}{(1-2s)^3}}{q}{\,0,\,\frac{1}{2}\,}{\frac{1}{2};\,\frac{1}{2}}{
\frac{x\,{\left( x+s \left( s-2 \right) \right) }^2}{{\left(\left(
1 - 2s \right) x+s^2 \right) }^2}}, \\ \label{isog4}
\heun{   s^4 }{-q(s-1)^4}{\,0,\,\frac{1}{2}\,}{\frac{1}{2};\,\frac{1}{2}}{x} \equal
\heun{\frac{(s+1)^4}{(s-1)^4}}{q}{\,0,\,\frac{1}{2}\,}{\frac{1}{2};\,\frac{1}{2}}{
\frac{x\,(s+1)^4\left( x+s^2 \right)^2}{(x-1)\left(x-s^4\right)\left(x-s^2\right)^2}}.
\end{eqnarray}
They correspond to generic isogenies of degree 3 and 4 between Legendre elliptic curves.
The $t$-values are related by algebraic equations of the modular curves corresponding
to the congruence subgroups $\Gamma_0(3)\cap\Gamma(2)$ and
$\Gamma_0(4)\cap\Gamma(2)$ of $PSL(2,\ZZ)$, respectively,
while the pull-back coverings are the isogeny transformations of the $u$-coordinate
of (\ref{legendre}), in $x$ rather than $u$.
Equivalent statements hold for isogeny transformations of any degree $D$.
Parametric quadratic transformations (\ref{heunquadr})--(\ref{heunquadr2})
applied to $\heunde{1/2,1/2,1/2,1/2}$ correspond to the generic isogeny of degree 2,
while quartic transformation (\ref{eq:quarticl})
then represents the multiplication by 2 map on (\ref{legendre}).
Both quartic transformations (\ref{eq:quarticl}) and (\ref{isog4}) are compositions
of two quadratic Heun-to-Heun transformations.

\subsection{Hypergeometric-to-Heun transformations with two parameters}

In the following subsections we present the possible Gauss-to-Heun
transformation formulas up to the  fractional-linear transformations
(\ref{frlin2f1}), (\ref{frlinhb})--(\ref{frlinha}).
As explained with the items \refpart{i}--\refpart{ii} in \S \ref{heunhpgid},
the number of different two-term identities is determined by the number of
singularities with different 
exponent differences in the same fiber
and the number of non-symmetric branching fibers.

\subsubsection{The transformation P15:
$\hpgde{1/2,\,\la,\,\lb}\stackrel{3}{\longleftarrow}\heunde{1/2,\,\la,2\la,\,3\lb}$}
\label{sec:cubicpq}

Up to frac\-tio\-nal-li\-near trans\-for\-mations, we have the
following identities:
\begin{eqnarray} \label{eq:c1}
\heun{\frac14}{\frac{9ab}4}{3a,\,3b}{\frac12;\,a+b+\frac12}{x}\equal
\hpg{2}{1}{a,\,b\,}{\frac12}{\,x(4x-3)^2},\\ \label{eq:c2}
\heun{\frac{1}{4}}{\widehat{q}_1}{3a,\,{3b}}{{\frac{3}{2};\,}{a+b+\frac{1}{6}}}{x}\equal
\left(1-\frac{4x}{3}\right)\,
\hpg{2}{1}{\! a+\frac{1}{3},\,b+\frac{1}{3}\,}{\frac{3}{2}}{\,x(4x-3)^2},
\quad\\ \label{eq:c3}
\heun{\frac34}
{\frac{27ab}{4}}{3a,\,3b}{a+b+\frac{1}{2};\,\frac{1}{2}}
{x}\equal \hpg{2}{1}{a,\,b\,}{a+b+\frac{1}{2}}{\,x(4x-3)^2},\\  \label{eq:c4}
\heun{-3}{0}{3a,\,3b}{2a+2b;\,\frac{1}{ 2}}{x}\equal
\hpg{2}{1}{a,\,b\,}{a+b+\frac{1}{2}}{\,\frac{x^2(x+3)}{4}},\\ \label{eq:c5}
\heun{\frac{4}{3}}{\widehat{q}_2}{{3a,\,}{2a+b}}{{3a+3b-\frac{1}{2};\,}{\frac{1}{2}}}{x}\equal
\left(1-\frac{3x}{4} \right)^{\!-2a}
\hpg{2}{1}{a,\,b\,}{a+b+\frac{1}{2}}{\,\frac{x^3}{(4-3x)^2}},
\end{eqnarray}
where $\widehat{q}_1=(9ab+3a+3b-1)/4$, $\widehat{q}_2=6a^2+6ab-a$. 
The five formulas represent the five non-equivalent choices for the 
exponent difference at $x=0$.
The choices for the local exponent at $x=0$ are  $1/2,-1/2,\la,2\la,3\lb$ respectively.
The first two identities are related by Lemma \ref{localx0}.
The arguments of the first four transformations are polynomials.
Note that the cubic argument in (\ref{eq:c3}) is the same as in (\ref{eq:c1})--(\ref{eq:c2})
even if the fiber for $x=0$ is different. However, the branching pattern in both fibers
and the branching order for $x=0$ is the same, so the same configuration of the
singular points $x=0$, $x=1$, $x=\infty$ is possible
(even if the local exponents at the respective points are different).
The argument in (\ref{eq:c4}) is related to $x(4x-3)^2$ by the affine
transformation $x\mapsto (x+3)/4$, giving us other point as $x=0$
on essentially the same covering.

\subsubsection{The transformation P19:
$\hpgde{1/3,\,\la,\,\lb}\stackrel{3}{\longleftarrow}\heunde{\la,\,2\la,\lb,\,2\lb}$}
Up to frac\-tio\-nal-li\-near trans\-for\-mations, we have the
following identities:
\begin{eqnarray}
\heun{9}{\widehat{q}_3\!}{3a,\,2a+b}{a+b+\frac{1}{3};2a-2b+1}{x} \!\!\equal\!\!
(1-x)^{-2a}\,\hpg{2}{1}{a,\,b\,}{a+b+\frac{1}{3}}{\,-\frac{x(x-9)^2}{27(x-1)^2}},\\ 
\heun{\frac{8}{9}}{\widehat{q}_4\!}{{3a,\,}{2a+b}}{{2a+2b-\frac{1}{3};\,}{a+b+\frac{1}{3}}}{x} \!\!\equal\!\!\!
\left(1-\frac{9x}{8}\right)^{\!-2a}\hpg{2}{1}{a,\,b\,}{\!a+b+\frac{1}{3}}{\frac{27x^2(x-1)}{(9x-8)^2}}\!,\qquad
\end{eqnarray}
where $\widehat{q}_3=18a^2-9ab+6a,$ $\widehat{q}_4=4a^2+4ab-2a/3$. 
The choice between $\alpha$ and $\beta$ for the 
exponent difference at $x=0$ gives identities related by fractional-linear transformations, 
just as the choice between $2\la$ and $2\lb$.
Hence we have only two transformation formulas.

\subsubsection{The transformation P20:
$\hpgde{1/2,\,\la,\,\lb}\stackrel{4}{\longleftarrow}\heunde{\la,\,3\la,\lb,\,3\lb}$}

Up to frac\-tio\-nal-li\-near trans\-for\-mations, we have the
following identities:
\begin{eqnarray} 
\heun{\frac{9}{8}}{\widehat{q}_5\!}{4a,\,3a+b}{3a+3b-\frac{1}{2};a+b+\frac{1}{2}}{x} \!\!\equal\!\!\!
\left(1-\frac{8x}{9}\right)^{\!-3a}\!\hpg{2}{1}{a,\,b\,}{\!a+b+\frac{1}{2}}{\frac{64x^3(x-1)}{(8x-9)^3}}\!,
\qquad\\
\heun{\!-\frac{1}{8}\!}{\widehat{q}_6\!}{4a,\,3a+b}{a+b+\frac{1}{2};3a+3b-\frac{1}{2}}{x} \!\!\equal\!\!
(1+8x)^{-3a}\hpg{2}{1}{a,\,b\,}{a+b+\frac{1}{2}}{\,\frac{64x(1-x)^3}{(8x+1)^3}},
\end{eqnarray}
where $\widehat{q}_5=9a^2+9ab-3a/2,$ $\widehat{q}_6=3a^2-5ab+3a/2$. 
The choice between $\alpha$ and $\beta$ for the 
exponent difference at $x=0$ gives identities related by fractional-linear transformations,
 just as the choice between $3\la$ and $3\lb$.
Hence we have only two transformation formulas.

\subsubsection{The transformation P51:
$\hpgde{1/3,\,\la,\,\lb}\stackrel{3}{\longleftarrow}\heunde{\la,\,\la,\la,\,3\lb}$}
\label{sec:j0cubic}

Up to frac\-tio\-nal-li\-near trans\-for\-mations, we have the
following identities:
\begin{eqnarray} \label{eq:cubc1}
\heun{-\omega}{3(1\!-\!\omega)ab}{3a,\,3b}{a\!+\!b\!+\!\frac{1}{3};a\!+\!b\!+\!\frac{1}{3}}{x} \!\!\!\equal\!\!\!
\hpg{2}{1}{\!a,\,b\,}{a+b+\frac{1}{3}}{3(2\omega\!+\!1)x(x-1)(x+\omega)\!}\!, \qquad\\
\heun{\omega+1}{3(\omega+2)ab}
{3a,\,a+b+\frac13}{3b;\,2a-b+\frac23}{x}  \!\!\!\equal\!\!\! 
\left({\textstyle1+\frac{\omega-1}3x}\right)^{-3a}
\hpg{2}{1}{\!a,a+\frac13}{b+\frac{2}{3}}{\frac{x^3}{(x\!-\!\omega\!-\!2)^3}\!}\!,\quad
\end{eqnarray}
where $\omega$ is the root of $\omega^2+\omega+1=0$.
The choices for the local exponent at $x=0$ are \mbox{$\alpha$ and $3\beta$}.
To relate the argument in (\ref{eq:cubc1}) to \cite[formula (3.6a)]{Maier}, note that
$$3(2\omega+1)x(x-1)(x+\omega)=1-\big(1-(\omega+2)x\big)^3.$$


\subsubsection{Two composite transformations (P2 and P3)}
\label{sec:compost}

As indicated in Table \ref{clasfig}, 
there are two composite Gauss-to-Heun transformations with two parameters.
They have degree 4, and transform $\hpgde{1/2,\,\la,\,\lb}$ to $\heunde{2\la,\,2\la,2\lb,\,2\lb}$
or $\heunde{\la,\,\la,2\la,\,4\lb}$.

The transformation \PR{2} can be expressed as a composition of two quadratic transformations
in three ways:
 \begin{eqnarray} \label{eq:compDD}
\PR2: && \hpgde{\fr12,\,\la,\,\lb}\pback{2}\hpgde{\la,\,\la,\,2\lb}
\pback{2}\heunde{2\la,\,2\la,\,2\lb,\,2\lb}, \nonumber\\
&& \hpgde{\fr12,\,\la,\,\lb}\pback{2}\hpgde{2\la,\,\lb,\,\lb}
\pback{2}\heunde{2\la,\,2\la,\,2\lb,\,2\lb}, \\
&& \hpgde{\fr12,\,\la,\,\lb}\pback{2}\heunde{1/2,1/2,\,2\la,\,2\lb}
\pback{\HT2}\heunde{2\la,\,2\la,\,2\lb,\,2\lb}. \nonumber
\end{eqnarray}
In the third expression, the transformation $\PR1$ is composed with 
Heun-to-Heun transformation (\ref{eq:quadtr3}).
Up to frac\-tio\-nal-li\-near trans\-for\-mations, we have one identity:
\begin{equation}
\heun{-1}{0}{{4a,\,}{2a-2b+1}}{{2a+2b;\,}{2a-2b+1}}{x}=(1-x^2)^{-2a}\hpg{2}{1}{a,\,b\,}{a+b+\frac{1}{2}}
{\,-\frac{4x^2}{(x^2-1)^2}},
\end{equation}
as the choice of the 
exponent differences $2\alpha$ or $2\beta$ for $x=0$ gives equivalent formulas.
The identity is a composition of (\ref{eq:quadr2}) and (\ref{H-2F1 1tr}).
The first two expressions in (\ref{eq:compDD}) imply a relation between $\hpgde{\la,\la,2\lb}$ 
and   $\hpgde{2\la,\lb,\lb}$, and hypergeometric identities such as
\begin{equation}
\hpg21{a, b}{2a}{x(2-x)}=(1-x)^{-b}\,\hpg21{2a-b,
b}{a+\frac12}{\frac{x^2}{4(x-1)}}.
\end{equation}
This is a bi-quadratic transformation with two free parameters. 
A few transformations of this kind  
are presented in \cite[p.~128--130]{SF}.

The transformation \PR{3} can be composed in one way:
\begin{eqnarray} \label{eq:comp35}
\PR3: && \hpgde{\fr12,\,\la,\,\lb}\pback{2}\hpgde{\la,\,\la,\,2\lb}
\pback{2}\heunde{\la,\,\la,\,2\la,\,4\lb}.
\end{eqnarray}
There are indeed three non-equivalent choices for the 
exponent difference at $x=0$, namely $\la,2\la,3\lb$.
Here are the respective formulas, up to frac\-tio\-nal-li\-near trans\-for\-mations:
\begin{eqnarray} \label{eq:q1}
\heun{\frac{1}{2}}{8ab}{{4a,\,}{4b}}{{a+b+\frac{1}{2};\,}{a+b+\frac{1}{2}}}{x}\equal
\hpg{2}{1}{a,\,b\,}{a+b+\frac{1}{2}}{ 16x(1-x)(1-2x)^2},\\   \label{eq:q2}
\heun{-1}{0}{{4a,\,}{4b}}{{2a+2b;\,}{a+b+\frac{1}{2}}}{x}\equal
\hpg{2}{1}{a,\,b\,}{a+b+\frac{1}{2}}{ 4x^2(1-x^2)},\\  \label{eq:q3}
\heun{-1}{0}{4a,\,2b}{4b-1;2a-b+1}{x} \equal
\left(1-\frac{x^2}2\right)^{\!-2a}
\hpg{2}{1}{a,\,a+\frac12\,}{b+\frac{1}{2}}{\frac{x^4}{(x^2-2)^2}}.\qquad
\end{eqnarray}
The three identities are compositions of, respectively,
(\ref{eq:quadr1}) and (\ref{H-2F1 3tr}), (\ref{eq:quadr1}) and (\ref{H-2F1 1tr}),
or (\ref{eq:quadr3}) and (\ref{H-2F1 1tr}).


\subsection{One-parameter transformations}
\label{eq:hpghe1p}

The one-parameter transformations are \PR{4}--\PR{14}, \PR{16}--\PR{18}, \PR{21}--\PR{50} and
\PR{52}--\PR{61}. This section exemplifies all {\em indecomposable} one-parameter transformations.
Composite transformations are less interesting. Especially, compositions with
Gauss-to-Gauss transformations do not affect the $t$ and $q$ parameters at all.
Appendix B spells out all compositions among the coverings \PR{1}--\PR{61},
and exemplifies the compositions \PR{9}, \PR{37}--\PR{39}, \PR{43} that
are obtained only by composing with Heun-to-Heun transformations.

Here are the indecomposable coverings, together with the covering and an illustrating formula
for each. In presenting formulas, we took a few pragmatic choices of style.
Most notably, we allow the argument of Heun functions to be a constant multiple of $x$,
so to avoid algebraic numbers (or longer expressions) in the coverings and on the right-hand sides
of our formulas. We also write some algebraic numbers
in denominators rather than numerators, when that makes a formula more compact.

\medskip \noindent
$\PR{5}:\hpgde{1/3,1/3,\,\la}\stackrel{4}{\longleftarrow}\heunde{\fr{1}{3},\fr{1}{3},\,\la,\,3\la}$,
with $\displaystyle\PH{46}:$ 
\begin{eqnarray*}
\heun{-1 }{\frac{4a(6a-1)}{3}}
{4a,\,4a+\frac{1}{3}}{\frac{2}{3};\,2a+\frac{2}{3} }{x}
=  \left( 1-2x\right)^{-3a }
\hpg{2}{1}{ a,\,a+\frac{1}{3}\,}{\frac{2}{3}}{-\frac{x(x-2)^3}{(2x-1)^3}};
\end{eqnarray*}
$\PR{25}:\hpgde{1/2,1/3,\,\la}\stackrel{5}{\longleftarrow}\heunde{\fr{1}{2},\fr{2}{3},\,2\la,\,3\la}$,
with $\displaystyle \PH{30}:$ 
\begin{eqnarray*}
\heun{\frac{1}{5}  }{\frac{2a}{3}}{5a,\,\frac{1}{2}-a}{\frac{1}{3}
;\, \frac{1}{2} }{x}= { \left( 1-5x\right)^{-2a }}
 \hpg{2}{1}{ a,\,\frac{1}{6}-a\,}{\frac{2}{3} }{\frac{x^2(9x-5)^3}{4(5x-1)^2}};
\end{eqnarray*}
$\PR{26}:\hpgde{1/2,1/3,\,\la}\stackrel{6}{\longleftarrow}\heunde{\fr{1}{3},\fr{2}{3},\,\la,\,5\la}$,
with $\displaystyle \PH{24}(x)=\frac{27x^2(x-1)(3x+125)^3}{4(9x-25)^5}$:
\begin{eqnarray*}
\heun{\frac{25}{9} }{\frac{5a}{3}}
{6a,\,4a+\frac{1}{6}}{\frac{1}{3} ;\,\frac{2}{3} }{x}
= \left( 1-\frac{9x}{25}\right)^{-5a}
\hpg{2}{1}{ a,\,\frac{1}{6}-a\,}{\frac{2}{3} }{\PH{24}(x) };
\end{eqnarray*}
$\PR{27}:\hpgde{1/2,1/4,\,\la}\stackrel{5}{\longleftarrow}\heunde{\fr{1}{2},\fr{1}{4},\,2\la,\,3\la}$,
with $\displaystyle \PH{29}(x)=\frac{x^2(x+80)^3}{(5x-32)^4}$:
\begin{eqnarray*}
\heun{-80 }{-25a(8a+1)}{5a,\,5a+\frac{1}{4}}
{4a+\frac{1}{2} ;\,\frac{1}{2} }{x}
= \left( 1-\frac{5x}{32}\right)^{-4a }
\hpg{2}{1}{ a,\,a+\frac{1}{4}\,}{2a+\frac{3}{4} }{\PH{29}(x)};
\end{eqnarray*}
$\PR{28}:\hpgde{1/2,1/3,\,\la}\stackrel{10}{\longleftarrow}\heunde{1/3,\,\la,\,4\la,\,5\la}$,
with $\displaystyle\PH9(x)=\frac{x \left(9x^3-90x^2+105x+40  \right)^3}
{64\left(x-9\right)\left(9x - 1 \right)^4}$:
\begin{eqnarray*}
\heun{\frac{1}{81}}{ \frac{50a(20a+1)}{81}}
{10a,\,\frac{5}{6} }{\frac{2}{3} ;\,2a+\frac{5}{6}}{\frac{x}{9}}
=\left(1-\frac{x}{9}\right)^{-a } \! \left(1-9x\right)^{-4a }
\hpg{2}{1}{ a,\, \frac{1}{6}-a\,}{\frac{2}{3} }{\PH9(x)};
\end{eqnarray*}
$\PR{29}:\hpgde{1/2,1/3,\,\la}\stackrel{8}{\longleftarrow}\heunde{\fr{2}{3},\la,\,2\la,\,5\la}$,
with $\displaystyle\PH{16}(x)=\frac{4x^2(x^2-8x+10)^3}{27(2x-1)^2(4x-27)}$:
\begin{eqnarray*}
\heun{\frac{2}{27} }{\frac{56a}{81}}
{8a,\,\frac{5}{6}-2a }{\frac{1}{3} ;\, 2a+\frac{5}{6} }{\frac{4x}{27}}
=\left(1-\frac{4x}{27}\right)^{-a} \! \left(1-2x\right)^{-2a}
\hpg{2}{1}{ a,\,\frac{1}{6}-a\,}{\frac{2}{3} }{\PH{16}(x) };
\end{eqnarray*}
$\PR{30}:\hpgde{1/2,1/3,\,\la}\stackrel{7}{\longleftarrow}\heunde{\fr{1}{2},\fr{1}{3},\,3\la,\,4\la}$,
with $\displaystyle\PH{23}(x)=-\frac{4x(27x^2-28x+7)^3}{(7x-4)^3} $:
\begin{eqnarray*}
\heun{\frac{27}{28} }{\frac{a(97-294a)}{24}}
{7a,\,\frac{2}{3}-a}{\frac{2}{3} ;\,\frac{1}{2} }{\frac{27x}{16}}
= \left(1-\frac{7x}{4}\right)^{-3a } 
\hpg{2}{1}{ a,\,\frac{1}{6}-a\,}{\frac{2}{3} }{\PH{23}(x)};
\end{eqnarray*}
$\PR{31}:\hpgde{1/2,1/3,\,\la}\stackrel{5}{\longleftarrow}\heunde{\fr{1}{2},\fr{2}{3},\,\la,\,4\la}$,
with $\displaystyle \PH{29}:$ 
\begin{eqnarray*}
\heun{\frac{32}{5} }{\frac{4a}{3}}{5a,\,3a+\frac{1}{6}}{\frac{1}{3} ;\,\frac{1}{2} }{x}
= \left( 1-\frac{5x}{32}\right)^{-4a}
 \hpg{2}{1}{ a,\,\frac{1}{6}-a\,}{\frac{2}{3} }{\frac{x^2(x+80)^3}{(5x-32)^4}};
\end{eqnarray*}
$\PR{32}:\hpgde{1/2,1/3,\,\la}\stackrel{10}{\longleftarrow}\heunde{1/3,\,2\la,\,3\la,\,5\la}$,
with $\displaystyle\PH{10}(x)=\frac{4x\left( 9x^3-60x^2+130x-90 \right)^3}
{9\left( 3x-8 \right)^2 \left( 4x-9 \right)^3}$:
\begin{eqnarray*}
\heun{\frac{27}{32} }{\!\frac{25a(11-30a)}{48}}
{10a,\,\frac{5}{6} }{\frac{2}{3} ;\,4a+\frac{2}{3}}{\frac{3x}{8}}
=\left(1-\frac{3x}{8}\right)^{\! -2a} \! \left(1-\frac{4x}{9}\right)^{\! -3a}
\hpg{2}{1}{ a,\,\frac{1}{6}-a\,}{\frac{2}{3} }{\PH{10}(x)};
\end{eqnarray*}
$\PR{33}:\hpgde{1/2,1/3,\,\la}\stackrel{9}{\longleftarrow}\heunde{1/2,\la,\,3\la,\,5\la}$,
with $\displaystyle\PH{13}(x)=\frac{27x(4x-3)^5}{4(x^3-12x^2-54x-2)^3}$:
\begin{eqnarray*}
\heun{\frac{3}{128} }{\frac{81a(1+51a)}{128}}
{9a,\,3a+\frac{1}{2} }{2a+\frac{5}{6};\,\frac12}{\frac{x}{32}}=
\left(1+27x+6x^2-\frac{x^3}{2}\right)^{-3a}
\hpg{2}{1}{ a,\, a+\frac{1}{3}\,}{2a+\frac{5}{6} }{\PH{13}(x)};
\end{eqnarray*}
$\PR{34}:\hpgde{1/2,1/5,\,\la}\stackrel{6}{\longleftarrow}\heunde{\fr{1}{5},\la,\,2\la,\,3\la}$,
with $\displaystyle\PH{24}(x)=\frac{27x^2(x-1)(3x+125)^3}{4(9x-25)^5}$:
\begin{eqnarray*}
\heun{-\frac{125}{3} }{-30a(1+10a)}
{6a,\,6a+\frac{1}{5}}{4a+\frac{2}{5};\,2a+\frac{7}{10} }{x}
= \left( 1-\frac{9x}{25}\right)^{-5a }
\hpg{2}{1}{ a,\,a+\frac{1}{5}\,}{2a+\frac{7}{10} }{\PH{24}(x) };
\end{eqnarray*}
$\PR{35}:\hpgde{1/2,1/3,\,\la}\stackrel{10}{\longleftarrow}\heunde{1/3,\,\la,\,2\la,\,7\la}$,
with $\displaystyle\PH8(x)=\frac{4x(x^3-12x^2+42x-42)^3}{27(4x-27)(3x-8)^2}$:
\begin{eqnarray*}
\heun{\frac{32}{81} }{ \! \frac{2a(179-686a)}{81}}
{10a,\,\frac{7}{6}-4a }{\frac{2}{3};\,2a+\frac{5}{6} }{\frac{4x}{27}}
= \left(1-\frac{4x}{27}\right)^{\!-a} \! \left(1-\frac{3x}{8}\right)^{\!-2a}
\hpg{2}{1}{a,\, \frac{1}{6}-a}{\frac{2}{3} }{\PH8(x)};
\end{eqnarray*}
$\PR{36}:\hpgde{1/2,1/3,\,\la}\stackrel{7}{\longleftarrow}\heunde{\fr{1}{2},\fr{1}{3},\,2\la,\,5\la}$,
with $\displaystyle \PH{22}(x)=\frac{4x(4x^2-35x+70)^3}{27(28x-125)^2}$:
\begin{eqnarray*}
\heun{\frac{125}{189} }{\frac{8a(38-147a)}{81}}
{7a,\,\frac{5}{6}-3a}{\frac{2}{3} ;\,\frac{1}{2} }{\frac{4x}{27}}
= \left( 1-\frac{28x}{125}\right)^{-2a} \, 
\hpg{2}{1}{ a,\,\frac{1}{6}-a\,}{\frac{2}{3} }{\PH{22}(x) };
\end{eqnarray*}
$\PR{40}:\hpgde{1/3,1/3,\,\la}\stackrel{4}{\longleftarrow}\heunde{\fr{1}{3},\fr{1}{3},\,2\la,\,2\la}$,
with $\displaystyle \PH{47}:$ 
\begin{eqnarray*}
\heun{15\sqrt3-26  }{\frac{8a(12a+1)}{3(5+3\sqrt{3})}}
{4a,\,4a+\frac{1}{3}}{\frac{2}{3};\,4a+\frac{1}{3} }{\frac{x}{5+3\sqrt{3}}}
= \left(1-2x\right)^{-3a } \,
\hpg{2}{1}{ a,\,a+\frac{1}{3}\,}{\frac{2}{3}}{\frac{x(x+4)^3}{4(2x-1)^3}};
\end{eqnarray*}
$\PR{42}:\hpgde{1/2,1/5,\,\la}\stackrel{5}{\longleftarrow}\heunde{\fr{1}{2},\la,\,2\la,\,2\la}$,
with $\displaystyle \PH{45}(x)=\frac{x(x^2-10x+5)^2}{(x+1)^5}$:
\begin{eqnarray*}
\heun{9+4\sqrt5}{\frac{5a(10a+1)}{10-4\sqrt5}}
{5a,\,3a+\frac{3}{10}}{\frac{1}{2}; \,4a+\frac{2}{5} }{\frac{5x}{5-2\sqrt5}}
= \left( 1+x\right)^{-5a} \hpg{2}{1}{ a,\,a+\frac{1}{5}\,}{\frac{1}{2} }{\PH{45}(x)};
\end{eqnarray*}
$\PR{44}:\hpgde{1/2,1/4,\,\la}\stackrel{6}{\longleftarrow}\heunde{\fr{1}{4},\fr{1}{4},\,\la,\,5\la}$,
with $\displaystyle \PH{42}(x)=\frac{256x^5}{(x+5)^4(5x^2+6x+5)}$:
\begin{eqnarray*}
\heun{\frac{-7+24i}{25}}{\! \frac{5a(40a-1)}{6-8i}}
{6a,\,4a+\frac{1}{4}}{10a-\frac{1}{4};\,\frac{3}{4} }{\frac{5x}{4i-3}}
=  \left( 1+\frac{x}{5}\right)^{\!-4a} \left( 1+\frac{6x}{5}+x^2\right)^{\!-a}
 \hpg{2}{1}{\!a,\,a+\frac{1}{4}}{2a+\frac{3}{4} }{\PH{42}(x)};
\end{eqnarray*}
$\PR{45}:\hpgde{1/2,1/3,\,\la}\stackrel{6}{\longleftarrow}\heunde{\fr{1}{2},\fr{1}{2},\,\la,\,5\la}$,
with $\displaystyle \PH{26}(x)=\frac{1728x}{(x^2+10x+5)^3}$:
\begin{eqnarray*}
\heun{\frac{117+44i}{125} }{\frac{a(49-228a)}{11-2i}}
{6a,\,\frac{5}{6}-4a}{2a+\frac{5}{6};\,\frac{1}{2} }{\frac{x}{2i-11}}
= \left(1+2x+\frac{x^2}{5}\right)^{-3a}  
\hpg{2}{1}{ a,\,a+\frac{1}{3}\,}{2a+\frac{5}{6}}{\PH{26}(x)};
\end{eqnarray*}
$\PR{46}:\hpgde{1/2,1/5,\,\la}\stackrel{6}{\longleftarrow}\heunde{\fr{1}{5},\la,\,\la,\,4\la}$,
with $\displaystyle \PH{42}(x)=\frac{x^4(25x^2-22x+5)}{4(2x-1)^5}$:
\begin{eqnarray*}
\heun{\frac{117+44i}{125}}{\frac{10a(40a-1)}{11-2i}}
{6a,\,6a+\frac{1}{5}}{8a-\frac{1}{5};\,2a+\frac{7}{10} }{\frac{25x}{11-2i}}
= \left( 1-2x\right)^{-5a }
\hpg{2}{1}{ a,\,a+\frac{1}{5}\,}{2a+\frac{7}{10} }{\PH{42}(x)};
\end{eqnarray*}
$\PR{47}:\hpgde{1/2,1/3,\,\la}\stackrel{4}{\longleftarrow}\heunde{\fr{1}{2},\fr{1}{2},\,\fr{1}{3},\,4\la}$,
with $\displaystyle\PH{36}:$ 
\begin{eqnarray*}
\heun{\frac{23+10\sqrt{-2}}{27}}{\frac{32a(1-6a)}{3(5-\sqrt{-2})}}
{4a,\,\frac{2}{3}-4a}{\frac{2}{3} ;\,\frac{1}{2} }{\frac{x}{5-\sqrt{-2}}}
=  \hpg{2}{1}{ a,\,\frac{1}{6}-a\,}{\frac{2}{3} }{-\frac{x(x-4)^3}{27}};
\end{eqnarray*}
$\PR{48}:\hpgde{1/2,1/4,\,\la}\stackrel{4}{\longleftarrow}\heunde{\fr{1}{2},\fr{1}{2},\,\la,\,3\la}$,
with $\displaystyle \PH{36}:$ 
\begin{eqnarray*}
\heun{\frac{17+56\sqrt{-2}}{81}}{\frac{a\,(17-40a)}{7-4\sqrt{-2}}}
{4a,\,\frac{3}{4}-2a}{2a+\frac{3}{4} ;\,\frac{1}{2} }{\frac{x}{-7+4\sqrt{-2}}}
= \left(1+\frac{x}{3} \right)^{-4a} \,
\hpg{2}{1}{ a,\,a+\frac{1}{4}\,}{2a+\frac{3}{4}}{\frac{256\,x}{(x+3)^4}};
\end{eqnarray*}
$\PR{49}:\hpgde{1/2,1/3,\,\la}\stackrel{10}{\longleftarrow}\heunde{1/3,\,\la,\,\la,\,8\la}$,
with $\displaystyle\PH7(x)=-\frac{4x(x^3-6x^2+15x-12)^3}{27\,(3x^2-14x+27)}$:
\begin{eqnarray*}
\heun{\frac{17+56\sqrt{-2}}{81}}{\frac{4a\,(13-64a)}{7-4\sqrt{-2}}}
{10a ,\,\frac{4}{3}-6a}{ \frac{2}{3};\, 2a+\frac{5}{6} }{\frac{3x}{7-4\sqrt{-2}}}
= \left(1-\frac{14x}{27}+\frac{x^2}{9} \right)^{-a}
\hpg{2}{1}{ a,\,\frac{1}{6}-a }{ \frac{2}{3}}{\PH7(x)};
\end{eqnarray*}
$\PR{50}:\hpgde{1/3,1/4,\,\la}\stackrel{4}{\longleftarrow}\heunde{\fr{1}{3},\la,\,\la,\,2\la}$,
with $\displaystyle \PH{36}:$ 
\begin{eqnarray*}
\heun{\frac{241+22\sqrt{-2}}{243}}{\frac{8a(7-8a)}{22-\sqrt{-2}}}
{4a,\,\frac{5}{6}}{\frac{2}{3} ;\,2a+\frac{7}{12} }{\frac{18x}{22-\sqrt{-2}}}
= \left(1-x\right)^{-4a} \,
 \hpg{2}{1}{ a,\,a+\frac{1}{4}\,}{\frac{2}{3} }{\frac{x(3x-4)^3}{27(x-1)^4}};
\end{eqnarray*}
$\PR{56}:\hpgde{1/2,1/3,\,\la}\stackrel{8}{\longleftarrow}\heunde{\fr{1}{3},\fr{1}{3},\,\la,\,7\la}$,
with $\displaystyle \PH{18}(x)=\frac{1728 x}{{(x^4- 14 x^3+63 x^2  - 70 x -7 ) }^2}$:
\begin{eqnarray*}
\heun{\frac{55+39\omega}{49} }{\!\frac{2a(71-348a)}{3(5-3\omega)}}
{8a,\,\frac76-6a }{2a+\frac{5}{6} ;\,\frac{2}{3} }{\frac{x}{5-3\omega}}
= \left(  1+10x-9x^2+2x^3-\frac{x^4}{7}\right)^{\!-2a}
\hpg{2}{1}{ \!a,\, a+\frac{1}{2}}{2a+\frac{5}{6} }{\PH{18}(x)};
\end{eqnarray*}
$\PR{57}:\hpgde{1/2,1/3,\,\la}\stackrel{9}{\longleftarrow}\heunde{1/2,\la,\,\la,\,7\la}$,
with $\displaystyle\PH{11}(x)=\frac{x \left( 2x^4-12x^3+42x^2-70x+63 \right)^2}
{27 \left( 4x^2-13x+32 \right) }$:
\begin{eqnarray*}
\heun{\frac{-87+91\sqrt{-7}}{256}}{\frac{2a(31-147a)}{13-7\sqrt{-7}}}
{9a,\,\frac{7}{6}-5a }{\frac{1}{2} ;\, 2a+\frac{5}{6}}{\frac{8x}{13-7\sqrt{-7}}}
= \left(1-\frac{13x}{32}+\frac{x^2}{8}\right)^{\!-a}
\hpg{2}{1}{ a,\,\frac{1}{6}-a\,}{\frac{1}{2} }{\PH{11}(x) };
\end{eqnarray*}
$\PR{58}:\hpgde{1/2,1/5,\,\la}\stackrel{5}{\longleftarrow}\heunde{\fr{1}{2},\la,\,\la,\,3\la}$,
with $\displaystyle \PH{37}(x)=\frac{x(x^2-10x+30)^2}{(x-4)^5}$:
\begin{eqnarray*}
\heun{\frac{781+171\sqrt{-15}}{1024}}{\frac{10a(23-45a)}{95-9\sqrt{-15}}}
{5a,\,\frac{9}{10}-a}{\frac{1}{2};\,2a+\frac{7}{10} }{\frac{20x}{95-9\sqrt{-15}}}
= \left( 1-\frac{x}4\right)^{-5a}
\hpg{2}{1}{ a,\,a+\frac{1}{5}\,}{\frac{1}{2} }{\PH{37}(x)};
\end{eqnarray*}
$\PR{59}:\hpgde{1/2,1/3,\,\la}\stackrel{5}{\longleftarrow}\heunde{\fr{1}{2},\fr{1}{3},\,\fr{1}{3},\,5\la}$,
with $\displaystyle \PH{37}:$ 
\begin{eqnarray*}
\heun{\frac{-7+33\sqrt{-15}}{128}}{\frac{25a(1-6a)}{11-3\sqrt{-15}}}
{5a,\,\frac{5}{6}-5a}{\frac{1}{2};\,\frac{2}{3} }{\frac{6x}{11-3\sqrt{-15}}}
= \hpg{2}{1}{ a,\,\frac{1}{6}-a\,}{\frac{1}{2} }{\frac{x(3x^2-10x+15)^2}{64}};
\end{eqnarray*}
$\PR{60}:\hpgde{1/2,1/4,\,\la}\stackrel{5}{\longleftarrow}\heunde{\fr{1}{2},\fr{1}{4},\,\la,\,4\la}$,
with $\displaystyle\PH{44}(x)=\frac{x(x+3-4i)^4}{(1+2i)^5(x-1)^4}$:
\begin{eqnarray*}
\heun{1+2i }{a\left(\frac{5}{4}+(7+24i)a\right)}
{5a,\,3a+\frac{1}{4}}{\frac{3}{4};\,8a }{x} = \left( 1-x\right)^{-4a}
\hpg{2}{1}{ a,\,\frac{1}{4}-a\,}{\frac{3}{4}}{\PH{44}(x)};
\end{eqnarray*}
$\PR{61}:\hpgde{1/2,1/3,\,\la}\stackrel{7}{\longleftarrow}\heunde{\fr{1}{2},\fr{1}{3},\,\la,\,6\la}$,
with $\displaystyle
\PH{21}(x)=\frac{4x(x^2-(5+4\omega)x+1+5\omega)^3}{3(\omega+2)((2\omega-1)x+9))}$:
\begin{eqnarray*}
\heun{\frac{3-12\omega}{7}}{\frac{2a(7-2(\omega-18)a)}{3+\omega}}
{7a,\,1-5a}{\frac{2}{3} ;\,\frac12}{\frac{x}{1+2\omega}}
=\left( 1-\frac{(1-2\omega)x}{9}\right)^{\!-a}
\hpg{2}{1}{ a,\, \frac{1}{6}-a\,}{\frac{2}{3} }{\PH{21}(x) }.
\end{eqnarray*}

\appendix
\section{Appendix: Fractional-linear transformations}

The hypergeometric equation (\ref{HGE}) with general $\A,\B,\C$
has the following local bases of solutions:
\begin{eqnarray*}
\mbox{at } z=0: &&\!\!\! \hpg{2}{1}{\A,\,\B\,}{\C}{\,z}, \quad
z^{1-\C}\;\hpg{2}{1}{\!1+\A-\C,\,1+\B-\C}{2-\C}{\,z};\\ \mbox{at }
z=1: &&\!\!\! \hpg{2}{1}{\A,\,\B\,}{1+\A+\B-\C}{1-z}, \quad
(1-z)^{\C-\A-\B}\hpg{2}{1}{\!\C-\A,\,\C-\B}{1+\C-\A-\B}{1-z};\\
\mbox{at } z=\infty: &&\!\!\!
z^{-\A}\;\hpg{2}{1}{\A,\,1+\A-\C\,}{1+\A-\B}{\,\frac{1}{z}}, \quad
z^{-\B}\;\hpg{2}{1}{\!\B,\,1+\B-\C}{1-\A+\B}{\,\frac{1}{z}}.
\end{eqnarray*}
The following Pfaff and Euler fractional-linear transformations
\cite[Th. 2.2.5]{SF} can be applied to the 6 local solutions:
\begin{eqnarray} \label{frlin2f1}
\hpg{2}{1}{\A,\,\B\,}{\C}{\,z}\equal(1-z)^{-\A}\;\hpg{2}{1}{\A,\,\C-\B\,}{\C}{\,\frac{z}{z-1}}
\nonumber\\
\equal(1-z)^{-\B}\;\hpg{2}{1}{\C-\A,\,\B\,}{\C}{\,\frac{z}{z-1}}\\
\equal(1-z)^{\C-\A-\B}\;\hpg{2}{1}{\C-\A,\,\C-\B\,}{\C}{\,z}.
\nonumber
\end{eqnarray}
This gives $6\times 4=24$ different hypergeometric solutions
for a general hypergeometric equation in total; 
they are referred to as the 24 {\em Kummer's solutions}.
The automorphism group of the hypergeometric equation
is the Coxeter group $\mathcal{A}_3$ of order 24 
 \cite{Dwork84}.
It contains the permutation group $\mathcal{S}_3$.
If the parameters $a,b$ are not considered as symmetric,
the automorphism group extends to a semidirect product
of $\mathcal{S}_3$ and $(\ZZ/2\ZZ)^3$, of order $3!\times 2^3=48$. The action
of this group is represented by a permutation of the 3 singular points and
interchange of the local exponents at those points. Accordingly, the group permutes
the local exponents \mbox{$(1-C,C-A-B,B-A)$} and multiplies (some of them) by $-1$.
The permutation of the singular points is realized by the M\"obius transformations
mapping $z$ to $z$, $1-z$, $1/z$, $1/(1-z)$, $z/(z-1)$ and $1-1/z$.
For integer values of the parameters $\A,\B$ or the local exponents,
the structure of Kummer's $24$ solutions degenerates \cite{Vidunas2007}.

The automorphism group of the Heun equation (\ref{Heun}) is the Coxeter group $\mathcal{D}_4$,
of order 192; see \cite{Maier192}. It contains the permutation group $\mathcal{S}_4$,
and extends to a semidirect product of $\mathcal{S}_4$ and $(\ZZ/2\ZZ)^4$ when
the parameters $a,b$ are distinguished.  Here are bases of local
solutions at $x=0$, $x=1$, $x=t$ and $x=\infty$ for a general Heun equation:
\begin{eqnarray} \label{eq:heunbasis}
\mbox{at } x=0: && \heun{t}{q}{a,\,b}{c;\,d}{x}, \qquad
x^{1-c}\,\heun{t}{\,q_1}{a-c+1,\,b-c+1}{2-c;\,d}{x}; \nonumber\\
\mbox{at } x=1: && \heun{1-t}{ab-q}{a,\;b}{d;\,c}{1-x}, \quad
(1-x)^{1-d}\,\heun{1-t}{q_2}{a-d+1,\,b-d+1}{2-d;\,c}{1-x}; \nonumber \\
\mbox{at } x=t:  &&
\heun{1-1/t}{ab-q/t}{a,\;b}{a+b-c-d+1;\,c}{1-\frac{x}t}, 
\\ && \hspace{91pt}
\left(1-\frac{x}t\right)^{c+d-a-b}\heun{1-1/t}{q_3}{c+d-a,c+d-b}{c+d-a-b+1;c}{1-\frac{x}t};
\nonumber\\ \hspace{-24pt}
 \mbox{at } x=\infty: &&x^{-a}\,\heun{1/t}{q_4}{a,\,a-c+1}{a-b+1;\,d}{\frac{1}{x}\,},
\qquad
x^{-b}\,\heun{1/t}{q_5}{b,\,b-c+1}{b-a+1;\,d}{\frac{1}{x}\,};
\nonumber
\end{eqnarray}
where
\begin{eqnarray*}
q_1 \equal q-(c-1)(a+b-c-d+d\,t+1),\\ q_2 \equal
a\,b-q-(d-1)(a+b-c\,t-d+1),\\
 q_3 \equal
a\,b-q/t+(c/t-c-d)(a+b-c-d),\\ q_4 \equal
q/t+a\,(a-b/t-c-d+d/t+1),\\ q_5 \equal q/t+b\,(b-a/t-c-d+d/t+1).
\end{eqnarray*}
Each of these functions can be expressed as the Heun series
with all six $t$-parameter values in (\ref{eq:allts}):
\begin{eqnarray} \label{frlinhb}
\heun{t}{q}{a,\,b}{c;\,d}{x} \equal
\heun{1/t}{q/t}{a,\;b}{c;\,a+b-c-d+1}{\frac{x}{t}\,} \nonumber\\
\equal
(1-x)^{-a}\,\heun{t/(t-1)}{(act-q)/(t-1)}{a,\,a-d+1}{c;\;a-b+1}{\frac{x}{x-1}}
\nonumber\\ \equal \left(1-\frac{x}t\right)^{-a}
\heun{1/(1-t)}{(q-ac)/(t-1)}{a,\,c+d-b}{c;\,a-b+1}{\frac{x}{x-t}}
\\ \equal \left(1-\frac{x}t\right)^{-a}
\heun{1-t}{ac-q}{a,\,c+d-b}{c;\;d}{\frac{(1-t)\,x}{x-t}} \nonumber
\\ \equal
(1-x)^{-a}\,\heun{1-1/t}{ac-q/t}{a,\,a-d+1}{c;\,a+b-c-d+1}{\frac{(t-1)\,x}{t\,(x-1)}}.
\nonumber
\end{eqnarray}
Besides, there are 4 transformations which do not change the argument $x$ nor
the parameter $t$:
\begin{eqnarray} \label{frlinha}
\heun{t}{q}{{a,\,}{b}}{c;\,d}{x} \equal
(1-x)^{1-d}\,\heun{t}{q-c\,(d-1)\,t}{a-d+1,\,b-d+1}{c;\,2-d}{x}
\nonumber\\ \equal \left(1-\frac{x}t\right)^{c+d-a-b}
\heun{t}{\,q_6}{c+d-a,\,c+d-b}{c;\;d}{x}\\ \equal
(1-x)^{1-d}\left(1-\frac{x}t\right)^{c+d-a-b}\heun{t}{\,q_7}{c-a+1,\,c-b+1}{c;\;2-d}{x},
\nonumber
\end{eqnarray}
where $q_6=q-c\,(a+b-c-d)$, $q_7=q-c\,(a+b-c-d+d\,t-t).$
In total, there are $6\times4=24$ two-term fractional-linear transformations of the Heun functions,
and $8\times24=192$ different Heun series solutions of a general Heun equation,
as described by Maier \cite{Maier192}. If the parameters $a,b$ are
distinguished, the full set of $2\times 192$ fractional-linear transformations is represented
by the permutation and the $-1$ action on the 
exponent differences \mbox{$(1-c,1-d,c+d-a-b,b-a)$}.

The two-term fractional-linear transformations (\ref{frlinhb}) and (\ref{frlinha}) fix the 
exponent difference at $x=0$ in this representation, characteristically.
The transformations in (\ref{frlinha}) represent
interchange of the local exponents at $x=1$ and at $x=t$.
By applying  (\ref{frlinhb}), all Heun functions in  (\ref{eq:heunbasis}) can be transformed 
to have the same $t$-parameter. In particular, we have these four functions
as solutions of the same Heun equation:
\begin{eqnarray} \label{frlintinv}
\heun{t}{q}{a,\,b}{c;\,d}{x}, \qquad
x^{-a}\,\heun{t}{t\,q_4}{a,\,a-c+1}{a+b-c-d+1;\,a-b+1}{\frac{t}{x}},  \nonumber \\ 
(x-t)^{-a}\, 
\heun{t}{q-(b-d)t}{a,\,c+d-b}{d;\,c}{\frac{t\,(x-1)}{x-t}},\\ 
(x-1)^{-a}\, 
\heun{t}{q_8}{a,\,a-d+1}{a+b-c-d+1;\;a-b+1}{\frac{x-t}{x-1}}, \nonumber
\end{eqnarray}
with $q_8=q+a(a-c-d+1)t$. They correspond to the permutations of the singularities in two pairs.
For example, the first two functions are related by the interchange of $x=0$ and $x=\infty$ and
the interchange of $x=1$ and $x=t$. Any 3 of the functions in (\ref{eq:heunbasis}) are related 
by a linear three-term {\em connection formula} (since the order of Heun's equation is 2),
though their coefficients are not known yet in general
(unlike for Kummer's solutions of the hypergeometric equation).

Transformations of the hypergeometric and Heun equations can
be conveniently presented as transformations of Riemann's $P$-symbols; for
example
\begin{eqnarray} \label{eq:ptransform}
P\left\{
\begin{array}{ccc} 0 & 1 &  \infty
\\ 0 & 0 & \A
\\ 1-\C & \C-\A-\B &  \B
\end{array} \;z\; \right\}
\equal (1-z)^{\C-\A-\B} P\left\{
\begin{array}{ccc} 0 & 1 &  \infty
\\ 0 & 0 & \C-\B
\\ 1-\C & \A+\B-\C &  \C-\A
\end{array} \;z\; \right\} \nonumber \\
\equal P\left\{
\begin{array}{ccc} 0 & 1 &  \infty
\\ 0 & 0 & \A
\\ \C-\A-\B & 1-\C &  \B
\end{array} \;1-z\; \right\}.
\end{eqnarray}

\section{Appendix: Composite transformations}
\label{sec:composite}

As Tables \ref{clasfig} and \ref{clasfig2} indicate, many parametric Heun-to-Gauss
reductions are compositions of lower degree transformations between the hypergeometric
and Heun equations. 28 are composite transformations out of the 61 listed transformations. Here we explain
the composition notation in Tables \ref{clasfig}, \ref{clasfig2}, and recount the transformations more
thoroughly.

The numbers in the decomposition notation denote the degree of component transformations.
The factor $2_H$ denotes the quadratic Heun-to-Heun transformation (\ref{eq:quadtr3})
discussed in \S \ref{Heun2Heun}. A few other indexed numbers denote particular
coverings of low degree:
$\CT3$ denotes the cyclic covering \PH{33} with the branching pattern \branch{3}{3}{1+1+1},
while $\AT4$ stands for the covering \PH{36} with the pattern \branch{4}{3+1}{2+1+1},
and $\BT4$ stands for the covering \PH{46} with the pattern \branch{3+1}{3+1}{3+1}.
The unindexed numbers 3 and 4 denote the frequent coverings \PH{34} (\branch{3}{2+1}{2+1})
and \PH{47} (\branch{3+1}{3+1}{2+2}), respectively. The product notation has to be followed
from right to left to trace the composition from the starting hypergeometric equation.
In a composition, exactly one factor represents an indecomposable Gauss-to-Heun
transformation; it is the first one from the left which is not $2_H$. The other factors to the right
represent pull-backs between hypergeometric equations.

In Tables \ref{heunred0} and \ref{heunred}, $2_H$ denotes an applicability of the
quadratic Heun-to-Heun transformation following the arrow in (\ref{eq:quadtr3}),
while $2^H$ denotes applicability of this quadratic transformation from left-to-right
in (\ref{eq:quadtr3}), and $4_H$ denotes applicability of the composite quartic
transformation  (\ref{eq:quadtr4}) following the arrows.

The product $\DDT$ in Tables \ref{clasfig} and \ref{clasfig2} denotes 
a composition of quadratic transformations that can be realized in multiple ways,
possibly including $\HT2$. Mainly, it indicates involvement of the degree 4 
transformation $\PR2$ realized by the covering \PH{31}. As presented in (\ref{eq:compDD}), 
the transformation \PR2 can be split into quadratic transformations in three ways. 
The same covering $\PH{31}$ realizes 
the quartic Heun-to-Heun transformation (\ref{eq:quadtr4}). 

The other composite transformation with 2 free parameters is \PR3,
realized by the Belyi covering $\PH{35}$. The composition is given in (\ref{eq:comp35}).
The same covering $\PH{35}$ realizes this composite transformation: 
\begin{equation} \label{eq:trp37}
\PR{37}: 
\hpgde{\fr12,\fr14,\,\la}\pback{2}\heunde{\fr12,\fr12,\fr12,\,2\la}
\pback{\HT2}\heunde{\fr12,\fr12,\,2\la,\,2\la}.
\end{equation}
The specialization $\beta=1/4$ of P2 gives a composite transformation
of the same appearance
$\hpgde{\fr12,\fr14,\,\la}$ $\pback{2}\heunde{\fr12,\fr12,\fr12,\,2\la}
\pback{\HT2}\heunde{\fr12,\fr12,\,2\la,\,2\la}$.
But the coverings are different, 
and  the pulled-back Heun equations have different sets of  $t$-parameters.

Here are the most complicated composition lattices, for the transformations $\PR{10}$ and $\PR{12}$.
They include specialized splittings (\ref{eq:compDD}) of \PR2, 
and are realized by the Belyi coverings $\PH{41}$ and $\PH{5}$, respectively:
\begin{eqnarray} \label{eq:trp10}
\begin{picture}(420,33)(0,18)
\put(0,20){P10: $\hpgde{\fr12,\fr14,\la}$}
\put(122,43){\vector(-2,-1){28}}  \put(126,40){$\hpgde{\fr12,\,\la,\,\la}$} 
\put(117,23){\vector(-1,0){25}}  \put(119,20){$\hpgde{\fr14,\fr14,\,2\la}$} 
\put(104,5){\vector(-1,1){13}} \put(105,0){$\heunde{\fr12,\fr12,\fr12,2\la}$} 
\put(216,28){\vector(-2,1){28}}   
\put(218,23){\vector(-1,0){25}}  
\put(216,19){\vector(-1,-1){12}}  
\put(220,20){$\heunde{\fr12,\fr12,2\la,2\la}$}  \put(221,0){$\heunde{\fr12,\fr12,2\la,2\la}$}
\put(205,1){$\Longleftarrow$} 
\put(240,40){$\hpgde{\la,\,\la,\,2\la}$} \put(238,43){\vector(-1,0){43}}  
\put(330,20){$\heunde{2\la,2\la,2\la,2\la}$} \put(329,23){\vector(-1,0){15}}
\put(329,17){\vector(-1,-1){13}}   
\put(330,29){\vector(-3,1){36}}   
\end{picture}
\\
\begin{picture}(420,56)(0,20)
\put(0,20){P12: $\hpgde{\fr12,\fr13,\la}$}
\put(121,43){\vector(-2,-1){30}}  \put(124,40){$\hpgde{\fr13,\fr13,2\la}$} 
\put(126,23){\vector(-1,0){31}}  \put(129,20){$\hpgde{\fr12,\la,\,2\la}$} \put(113,26){\scriptsize 3}
\put(223,38){\vector(-3,-1){30}}  \put(226,40){$\hpgde{2\la,\,2\la,\,2\la}$} 
\put(228,22){\vector(-1,0){34}}  \put(231,20){$\hpgde{\la,\,\la,\,4\la}$} 
\put(206,5){\vector(-1,1){13}} \put(208,0){$\heunde{\fr12,\fr12,2\la,4\la}$} 
\put(316,28){\vector(-2,1){26}}  
\put(317,23){\vector(-1,0){30}} 
\put(317,17){\vector(-1,-1){13}} 
\put(319,20){$\heunde{2\la,2\la,4\la,4\la}$}
\put(224,43){\vector(-1,0){29}} \put(205,47){\scriptsize $3_C$}
\end{picture}
\end{eqnarray}
\vspace{8pt}

\noindent
The degree indications $2$ and $2_H$ are not shown with the arrows, only $3$ and $3_C$.
The two Heun equations  $\heunde{\fr12,\fr12,2\la,2\la}$ in the P10 lattice are different,
indicating that both \PR2 and \PR{37} appear as composition factors of P10.
The quadratic transformations $\heunde{\fr12,\fr12,\fr12,2\la}\pback{2_H}\heunde{\fr12,\fr12,2\la,2\la}$
branch over 2 of the 3 singularities with the exponent difference $1/2$, giving three choices.
One choice leads to a specialized P2 component
$\hpgde{\fr12,\fr14,\,\la}\pback4\heunde{\fr12,\fr12,\,2\la,\,2\la}$,
and two choices lead to (\ref{eq:trp37}) as indicated by the double arrow
in the lower part of (\ref{eq:trp10}).
This sheds more light on the distinction between \PR2 and \PR37 in general. 

Formally, there are always three choices for the first leg of the quartic 
Heun-to-Heun transformation (\ref{eq:quadtr4}). This quartic 
transformation is involved 
as a component of \PR{38} and \PR{53} as well.

Here are 
all compositions for the parametric Gauss-to-Heun transformations in 
Tables \ref{clasfig} and \ref{clasfig2} (and their coverings),
except for the just considered \PR2, \PR3, \PR{10}, \PR{12}, \PR{37}.
\begin{eqnarray*}
\PR4 \, (\PH{27}):
&& \hpgde{\fr12,\fr13,\,\la}\pback{3}\hpgde{\fr12,\,\la,\,2\la}
\pback{2}\heunde{\fr12,\fr12,\,2\la,\,4\la}; \hspace{88pt}  \\
\PR{6} \, (\PH{19}):
&& \hpgde{\fr12,\fr13,\,\la}\pback{4}\hpgde{\fr13,\,\la,\,3\la}
\pback{2}\heunde{\fr13,\fr13,\,2\la,\,6\la}; \\
\PR{7} \, (\PH{15}):
&& \hpgde{\fr12,\fr13,\,\la}\pback{4}\hpgde{\fr13,\,\la,\,3\la}
\pback{2}\heunde{\fr23,\,\la,\,\la,\,6\la}; \\
\PR{8} \, (\PH{17}):
&& \hpgde{\fr12,\fr13,\,\la}\pback{4}\hpgde{\fr13,\,\la,\,3\la}
\pback{2}\heunde{\fr23,\,2\la,\,3\la,\,3\la}; \\
\PR9 \, (\PH{48}):
&& \hpgde{\fr14,\fr14,\,\la}\pback{2}\heunde{\fr12,\fr12,\,\la,\,\la}
\pback{2_H}\heunde{\la,\,\la,\,\la,\,\la}; \\
\PR{11} \, (\PH{28}):
&& \hpgde{\fr13,\fr13,\,\la}\pback{3_C}\hpgde{\la,\,\la,\,\la}
\pback{2}\heunde{\la,\,\la,\,2\la,\,2\la}; \\
\end{eqnarray*}
\begin{eqnarray*}
\PR{13} \, (\PH{40}):
&& \hpgde{\fr12,\fr14,\,\la}\pback{2}\hpgde{\fr12,\,\la,\,\la}
\pback{2}\hpgde{\la,\,\la,\,2\la}\pback{2}\heunde{\la,\,\la,\,2\la,\,4\la}; \\
\PR{14} \, (\PH{2}): &&
\hpgde{\fr12,\fr13,\,\la}\pback{3}\hpgde{\fr12,\,\la,\,2\la}
\pback{2}\hpgde{\la,\,\la,\,4\la}\pback{2}\heunde{\la,\,\la,\,2\la,\,8\la}; \\
\PR{16} \, (\PH{25}):
&& \hpgde{\fr12,\fr14,\,\la}\pback{2}\hpgde{\fr12,\,\la,\,\la}
\pback{3}\heunde{\fr12,\,\la,\,2\la,\,3\la}; \\
\PR{17} \, (\PH{12}):
&& \hpgde{\fr12,\fr13,\,\la}\pback{3}\hpgde{\fr12,\,\la,\,2\la}
\pback{3}\heunde{\fr12,\,\la,\,2\la,\,6\la}; \\
\PR{18} \, (\PH{14}):
&& \hpgde{\fr12,\fr13,\,\la}\pback{3}\hpgde{\fr12,\,\la,\,2\la}
\pback{3}\heunde{\fr12,\,2\la,\,3\la,\,4\la}, \\
&& \hpgde{\fr12,\fr13,\,\la}\pback{2}\hpgde{\fr13,\fr13,\,2\la}
\pback{\BT4}\heunde{\fr13,\fr13,\,2\la,\,6\la}; \\
\PR{21} \, (\PH{25}):
&& \hpgde{\fr12,\fr13,\,\la}\pback{2}\hpgde{\fr13,\fr13,\,2\la}
\pback{3}\heunde{\fr13,\fr23,\,2\la,\,4\la}; \\
\PR{22} \, (\PH{39}):
&& \hpgde{\fr12,\fr16,\,\la}\pback{2}\hpgde{\fr13,\,\la,\,\la}
\pback{3}\heunde{\la,\,\la,\,2\la,\,2\la}, \\
&& \hpgde{\fr12,\fr16,\,\la}\pback{3}\heunde{\fr12,\fr12,\,\la,\,2\la}
\pback{\HT2}\heunde{\la,\,\la,\,2\la,\,2\la}; \\
\PR{23} \, (\PH{20}):
&& \hpgde{\fr12,\fr14,\,\la}\pback{2}\hpgde{\fr12,\,\la,\,\la}
\pback{4}\heunde{\la,\,\la,\,3\la,\,3\la}, \\
&& \hpgde{\fr12,\fr14,\,\la}\pback{\AT4}\heunde{\fr12,\fr12,\,\la,\,3\la}
\pback{2_H}\heunde{\la,\,\la,\,3\la,\,3\la}; \\
\PR{24} \, (\PH{3}): &&
\hpgde{\fr12,\fr13,\,\la}\pback{3}\hpgde{\fr12,\,\la,\,2\la}
\pback{4}\heunde{\la,\,2\la,\,3\la,\,6\la}, \\ &&
\hpgde{\fr12,\fr13,\,\la}\pback{4}\hpgde{\fr13,\,\la,\,3\la}
\pback{3}\heunde{\la,\,2\la,\,3\la,\,6\la}; \\
\PR{38} \, (\PH{28}):
&& \hpgde{\fr12,\fr13,\,\la}\pback{3_C}\heunde{\fr12,\fr12,\fr12,\,3\la}
\pback{2_H}\heunde{\fr12,\fr12,\,3\la,\,3\la}; \\
\PR{39} \, (\PH{43}):
&& \hpgde{\fr12,\fr14,\,\la}\pback{3}\heunde{\fr12,\fr12,\fr14,\,3\la}
\pback{2_H}\heunde{\fr14,\fr14,\,3\la,\,3\la}; \\
\PR{41} \, (\PH{20}):
&& \hpgde{\fr12,\fr13,\,\la}\pback{2}\hpgde{\fr13,\fr13,\,2\la}
\pback{4}\heunde{\fr13,\fr13,\,4\la,\,4\la}, \\
&& \hpgde{\fr12,\fr13,\,\la}\pback{\AT4}\heunde{\fr12,\fr12,\fr13,\,4\la}
\pback{2_H}\heunde{\fr13,\fr13,\,4\la,\,4\la}; \\
\PR{43} \, (\PH{4}): &&
\hpgde{\fr12,\fr13,\,\la}\pback{6}\heunde{\fr12,\fr12,\,\la,\,5\la}
\pback{2_H}\heunde{\la,\,\la,\,5\la,\,5\la}; \\
\PR{52} \, (\PH{38}):
&& \hpgde{\fr12,\fr13,\,\la}\pback{2}\hpgde{\fr13,\fr13,\,2\la}
\pback{3_C}\heunde{\fr13,\fr13,\fr13,\,6\la}; \\
\PR{53} \, (\PH{6}):
&& \hpgde{\fr12,\fr13,\,\la}\pback{4}\hpgde{\fr13,\,\la,\,3\la}
\pback{3_C}\heunde{3\la,\,3\la,\,3\la,\,3\la}, \\
&& \hspace{-16pt} \hpgde{\fr12,\fr13,\la}\pback{3_C}\heunde{\fr12,\fr12,\fr12,3\la}
\pback{2_H}\heunde{\fr12,\fr12,3\la,3\la}\pback{2_H}\heunde{3\la,3\la,3\la,3\la}; \\
\PR{54} \, (\PH{38}):
&& \hpgde{\fr12,\fr16,\,\la}\pback{2}\hpgde{\fr13,\,\la,\,\la}
\pback{3_C}\heunde{\la,\,\la,\,\la,\,3\la}; \\
\PR{55} \, (\PH{1}): &&
\hpgde{\fr12,\fr13,\,\la}\pback{4}\hpgde{\fr13,\,\la,\,3\la}
\pback{3_C}\heunde{\la,\,\la,\,\la,\,9\la}.
\end{eqnarray*}


Finally, we present a few exemplifying formulas for the compositions with $2_H$
that cannot be obtained by composing with Gauss-to-Gauss transformations.
The additional Heun-to-Heun transformation changes the $t$-parameter except for P9.
\\

\noindent
$\PR{9}:\hpgde{1/4,1/4,\,\la}\stackrel{4}{\longleftarrow}\heunde{\la,\,\la,\,\la,\,\la}$,
with $\displaystyle \PH{48}$:
\begin{eqnarray*}
  \heun{-1}{0} {4a,\,2a+\frac{1}{2}}{2a+\frac{1}{2}
 ;\,2a+\frac{1}{2} }{x} = \left( 1-ix \right)^{-4a} 
 \hpg{2}{1}{ a,\,a+\frac{1}{4}\,}{2a+\frac{1}{2}} 
 {\frac{8ix(x^2-1)}{(x+i)^4}};
\end{eqnarray*}
\\
 $\PR{37}:\hpgde{1/2,1/4,\,\la}\stackrel{4}{\longleftarrow}\heunde{1/2,1/2,\,2\la,\,2\la}$,
with $\displaystyle \PH{35}$:
\begin{eqnarray*}
\heun{17+12\sqrt{2}}{\frac{2a(1+8a)}{3-2\sqrt2}}
{4a,\,4a+\frac{1}{2}}{\frac{1}{2};\,4a+\frac{1}{2} } {\frac{x}{3-2\sqrt{2}}}
= \left(1+x\right)^{-4a} 
\hpg{2}{1}{ a,\,a+\frac{1}{4}\,}{\frac{1}{2}}{\frac{16x(x-1)^2}{(x+1)^4}};
\end{eqnarray*}
 \\
 $\PR{38}:\hpgde{1/2,1/3,\,\la}\stackrel{6}{\longleftarrow}\heunde{1/2,1/2,\,3\la,\,3\la}$,
with $\displaystyle \PH{28}(x)=\frac{36x\left(x^2+3\right)^2}{(x^2+6x-3)^3}$:
\begin{eqnarray*}
  \heun{4\sqrt3-7}{\frac{3a(1+12a)}{3+2\sqrt3}}
 {6a,\,6a+\frac{1}{2}}{\frac{1}{2} ;\,6a+\frac{1}{2} }{\frac{x}{3+2\sqrt{3}}}=
  \left(1-2x-\frac{x^2}{3}\right)^{-3a}
 \hpg{2}{1}{ a,\,a+\frac{1}{3}\,}{\frac{1}{2} }{\PH{28}(x)};
\end{eqnarray*}
 \\
$\PR{39}:\hpgde{1/2,1/4,\,\la}\stackrel{6}{\longleftarrow}\heunde{1/4,1/4,\,3\la,\,3\la}$,
with $\displaystyle \PH{43}(x)=\frac{108x(x-1)^4}{(x^2+14x+1)^3}$:
\begin{eqnarray*}
  \heun{97+56\sqrt3}{\frac{9a(1+24a)}{14-8\sqrt3}}
 {6a,\,6a+\frac{1}{4}}{\frac{3}{4};\,6a+\frac{1}{4} }{\frac{x}{4\sqrt{3}-7}}=
 \left(1+14x+x^2 \right)^{-3a}
 \hpg{2}{1}{ a,\,\frac{1}{4}-a}{\frac{3}{4}}{\PH{43}(x)};
\end{eqnarray*}
 \\
$\PR{43}:\hpgde{1/2,1/3,\,\la}\stackrel{12}{\longleftarrow}\heunde{\la,\,\la,\,5\la,\,5\la}$,
with $\displaystyle\PH{4}(x)=\frac{1728x^5(x^2-11x-1)}{(x^4-12x^3+14x^2+12x+1)^3}$:
\begin{eqnarray*}
\! \heun{\! \frac{-123+55\sqrt{5}}{2}}{\frac{12a(1+60a)}{11+5\sqrt{5}}}
{12a,\,2a+\frac{5}{6}}{10a+\frac{1}{6};2a+\frac{5}{6} }{\frac{2x}{11\!+\!5\sqrt{5}}}
\! =\! \left(1+12x+14x^2-12x^3+x^4\right)^{\!-3a}
 \hpg{2}{1}{ \! a,a+\frac{1}{3}}{2a+\frac{5}{6}}{\PH{4}(x)\!}.
\end{eqnarray*}
Note that the transformation P9 is not defined over $\QQ$ even if $t\in\QQ$.
Other example of this type is P11, with the same $t=-1$:
\begin{equation}
  \heun{-1}{0} {6a,\,2a+\frac{2}{3}}{4a+\frac{1}{3}
 ;\,2a+\frac{2}{3} }{x} = \left( 1-(\omega+1)x^2 \right)^{-3a} 
 \hpg{2}{1}{ a,\,a+\frac{1}{3}\,}{2a+\frac{2}{3}} 
 {\frac{3(1+2\omega)x^2(x^2-1)}{(x^2+\omega)^3}}.
\end{equation}
Here the covering is \PH{28}, the same as in \PR{38} (just above) but normalized differently.
Composition with $2_H$ 
occurs in \PR{38} but not in
$\PR{11}:\hpgde{1/3,1/3,\,\la}\stackrel{6}{\longleftarrow}\heunde{\la,\,\la,\,2\la,\,2\la}$.

\section{Appendix: Invariants of fractional-linear transformations}

Here we consider the action of the fractional-linear transformations of Appendix A
on Heun equations and functions, and give some invariants of it.
A the end, Theorem \ref{th:tinv} proves sufficient conditions for Heun's equation 
to be reducible to a hypergeometric equation by the considered parametric transformations. 

Additional invariants of the fractional-linear transformations are needed not only to determine
the accessory parameter $q$, but also to ensure that the $t$-value is in a right correspondence
with the assignment of local exponent differences to the four singular points. 
For example, the permutation of the singularities $x=0$ and $x=1$ 
changes the $t$-value to $1-t$;  hence we have
$\heunoppa{t}{q}{a,b}{c;d}{x}$ and \mbox{$\heunoppa{1-t}{ab-q}{a,b}{d;c}{1-x}$}
in the same orbit, but generally not $\heunoppa{t}{\tilde{q}}{a,b}{d;c}{\tilde{x}}$ 
for some $\tilde{q},\tilde{x}$.

When talking about equal (or different) local exponent differences in this appendix,
we mean equal up to multiplication by $-1$ (or different even after multiplication of some by $-1$).

\begin{theorem} \label{th:invs}
Consider Heun's equation $E_0$ as in $(\ref{Heun})$, and let 
\begin{equation}
e_0=1-c, \qquad e_1=1-d, \qquad e_t=c+d-a-b, \qquad e_{\infty}=b-a
\end{equation}
denote the local exponent differences.
\begin{enumerate}
\item If some two local exponent differences are equal,
there is a fractional-linear transformation of $E_0$ 
with the same parameters $a,b,c,d$, but with a different $t$-parameter.
\item If three local exponent differences are equal,
there are fractional-linear transformations of $E_0$
with the same parameters $a,b,c,d$ and any $t$ in the orbit $(\ref{eq:allts})$.
\item The following entities are invariants of the action of the fractional-linear transformations
on Heun equations (and functions):
\begin{itemize}
\item The elementary symmetric functions $E_1,E_2,E_3,E_4$ in 
the squares $e_0^2,e_1^2,e_t^2,e_{\infty}^2$, determined by the polynomial identity
\begin{equation}
X^4-E_1X^3+E_2X^2-E_3X+E_4=(X-e_0^2)(X-e_1^2)(X-e_t^2)(X-e_\infty^2).
\end{equation}
\item The $j$-invariant $j(t)$ as in $(\ref{eq:invj0})$.
\item If $j\neq 0$, the values
\begin{eqnarray} \label{eq:invk1}
k_1 \equal \frac{t^2-t+1}{t(t-1)}
\left((e_0^2-e_1^2)(e_t^2-e_\infty^2)\,t-(e_0^2-e_t^2)(e_1^2-e_\infty^2)\right),\\
\label{eq:invk2}
k_2 \equal \frac{(e_0^2e_t^2+e_1^2e_\infty^2)\,t^2
-(e_0^2e_\infty^2+e_1^2e_t^2)t+e_0^2e_1^2+e_t^2e_\infty^2}{t^2-t+1}+e_0^2e_\infty^2+e_1^2e_t^2.
\end{eqnarray}
\item 
The value $Q_1=J_1Q_0$, where
\begin{eqnarray} \label{eq:invj1}
\hspace{-20pt} J_1 \equal \frac{(t+1)\,(t-2)\,(2t-1)\,(t^2-t+1)}{t^2\,(t-1)^2},\\ 
\hspace{-20pt} \nonumber 
Q_0\equal  12q-6\,(e_0-1)\,(e_1\,t+e_t)
+\left(e_{\infty}^2-2(e_0-1)(e_0-2)\right)(t+1)-e_1^2\,(2t-1)+e_t^2\,(t-2).
\end{eqnarray}
\end{itemize}
\item The invariants $j,k_1$ determine the $t$-values corresponding to 
an orderly assignment  $(e_0,e_1,e_t,e_\infty)$ of the exponent differences
to the singular points, except when $j\in\{0,1728\}$ or
\begin{equation} \label{eq:nodsing}
j=\frac{1728\,F_4^3}{(E_2^3-9F_6)^2}, \qquad
k_1=-\frac32\,\frac{F_4^2}{E_2^3-9F_6},
\end{equation}
with $F_4=E_2^2-3E_1E_3+12E_4$,
$F_6=\frac12E_1E_2E_3-\frac32E_3^2-\frac32E_1^2E_4+4E_2E_4$.
\item If a pair of local exponent differences is equal, 
the exceptional case in \refpart{d} has $j=1728$.
\item Algebraic relations between $j,k_1,k_2$ are generated by these generic
identifications of the $j$-invariant:
\begin{eqnarray} \label{eq:jk1k2}
\frac{j}{256}=\frac{-k_1^2+9k_1k_2-6E_2k_1-9F_4}{3k_2^2-4E_2k_2+4(E_1E_3-4E_4)}
=k_1\,\frac{3k_1k_2-2E_2k_1-3F_4}{F_4k_2-2\widetilde{F}_6},
\end{eqnarray}
where $\widetilde{F}_6=\frac12E_1E_2E_3-\frac92E_3^2-\frac92E_1^2E_4+16E_2E_4$.
\item The invariants $k_1,k_2$ determine the $t$-values corresponding to 
an orderly assignment  $(e_0,e_1,e_t,e_\infty)$ of the exponent differences.
\item If all local exponent differences are different, then $t$ is unique:
\begin{equation} \label{eq:tk1k2}
t=\frac12+\frac{k_1\left(k_2-(e_0^2+e_t^2)(e_1^2+e_\infty^2)\right)
\left(k_2-(e_0^2+e_\infty^2)(e_1^2+e_t^2)\right)-F_4(k_2-e_0^2e_1^2-e_t^2e_\infty^2)+\widetilde{F}_6}
{(e_0^2-e_1^2)(e_0^2-e_t^2)(e_0^2-e_\infty^2)(e_1^2-e_t^2)(e_1^2-e_\infty^2)(e_t^2-e_\infty^2)}.
\end{equation} 
If there is exactly one pair of equal exponent differences, there are two corresponding $t$-values.
They are determined by {\rm (\ref{eq:invk1})} or by {\rm (\ref{eq:invk2})}.
\end{enumerate}
\end{theorem}
\begin{proof} As mentioned in Appendix A, the fractional-linear transformations 
permute the exponent differences and multiply them by $-1$. The transformations
that leave $t$ invariant are those that multiply the exponent differences by $-1$ and
interchange them in two pairs; see (\ref{frlinha}) and (\ref{frlintinv}).
These transformations leave $e_0^2e_1^2+e_t^2e_\infty^2$, $e_0^2e_t^2+e_1^2e_\infty^2$, 
$e_0^2e_\infty^2+e_1^2e_t^2$ invariant as well.

Part \refpart{a} is demonstrated by the first three equalities in (\ref{frlinha}) 
in the cases $e_1=e_t$, $e_1=e_\infty$, $e_t=e_\infty$, respectively. 
Part \refpart{b} is demonstrated by all formulas (\ref{frlinha}) in the case $e_1=e_t=e_\infty$. 
Other possible equalities of exponent differences are obtained
by applying the transformations that leave $t$ invariant. 

The invariants $E_1,E_2,E_3,E_4$ and $j(t)$ are clear. To obtain other invariants 
that do not involve $q$, we consider the shortened orbit sums
\begin{eqnarray*}
S_1\equal (e_0^2e_1^2+e_t^2e_\infty^2)\left(\frac1t+\frac1{1-t}\right)
+(e_0^2e_t^2+e_1^2e_\infty^2)\left(t+\frac{t}{t-1}\right)
+(e_0^2e_\infty^2+e_1^2e_t^2)\left(1-t+\frac{t-1}{t}\right),\\
S_2\equal (e_0^2e_1^2+e_t^2e_\infty^2) \!\left( \frac1{t^2}+\frac1{(1-t)^2} \!\right)\!
+(e_0^2e_t^2+e_1^2e_\infty^2) \!\left(\! t^2\!+\frac{t^2}{(t-1)^2} \!\right)\!
+(e_0^2e_\infty^2+e_1^2e_t^2) \!\left(\! (1-t)^2\!+\frac{(t-1)^2}{t^2}\right)\!.
\end{eqnarray*}
Then $k_1=S_1-E_2$, $k_2=256(S_2+E_2)/j$, adjusted for brevity.
To obtain $Q_1$, let $S_3$ denote the full orbit sum of $192\times2$ values of the product $t\,q$.
Then
\begin{equation}
Q_1=\frac{3}{16}S_3-\frac{j(E_1+8)}{256}.
\end{equation}
To check the invariance of $Q_1$ directly,
it is useful to note this general action of the fractional-linear transformations
on the ``semi-invariant" $Q_0$:
\begin{equation} \label{eq:q0act}
\begin{array}{ll}
Q_0\mapsto Q_0, & \mbox{if $t$ remains the same};\\
Q_0\mapsto -Q_0, & \mbox{if $t$ is transformed to $1-t$};\\
Q_0\mapsto Q_0/t, & \mbox{if $t$ is transformed to $1/t$};\\
Q_0\mapsto -Q_0/t, & \mbox{if $t$ is transformed to $(t-1)/t$};\\
Q_0\mapsto Q_0/(1-t), & \mbox{if $t$ is transformed to $t/(t-1)$};\\
Q_0\mapsto Q_0/(t-1), & \mbox{if $t$ is transformed to $1/(1-t)$}.
\end{array}
\end{equation}

Now consider part \refpart{d}. The algebraic relation between $j,k_1$ is obtained by eliminating $t$
from (\ref{eq:invj0}) and (\ref{eq:invk1}). It is of degree 6 in $k_1$, naturally. 
Computation shows that the discriminant with respect to $k_1$ vanishes only for $j\in\{0,1728\}$
and for the $j$-value in (\ref{eq:nodsing}). Non-vanishing discriminant gives a one-to-one correspondence
between the $t$ and $k_1$ values for the same $j$-invariant. The ambiguous case
(\ref{eq:nodsing}) represents a nodal singularity on the plane algebraic curve
defined by the relation between $j,k_1$. It does not distinguish the following $t$-values:
\begin{equation} \label{eq:t1t2}
t_1=-\frac{E_2-3e_0^2e_\infty^2-3e_1^2e_t^2}{E_2-3e_0^2e_t^2-3e_1^2e_\infty^2}, \qquad
t_2=-\frac{E_2-3e_0^2e_1^2-3e_t^2e_\infty^2}{E_2-3e_0^2e_\infty^2-3e_1^2e_t^2}.
\end{equation}
If a pair of local exponent differences is equal, these $t$-values are in $\{-1,2,1/2\}$,
showing part \refpart{e}.

Part \refpart{f} and formula (\ref{eq:tk1k2}) follow from Gr\"obner basis computations that 
eliminate $t$ and $j$. Evidently, $t$ is determined uniquely by the ordered tuple $(e_0,e_1,e_t,e_\infty)$
when the local exponent differences  are not equal, and there are at least two fitting $t$-values 
when there is an equality by part \refpart{a}. Formula (\ref{eq:invk1}) becomes quadratic in $t$
in the latter case. Parts \refpart{g} and \refpart{h} follow.
\end{proof}

The invariant $Q_1$ clearly determines $q$ unless $j\in\{0,1728\}$.
On the other hand, the encountered Heun equations with $j\in\{0,1728\}$ 
all have $Q_0=0$, thus investigation of additional invariants is not needed.
Generally, the expression $Q_0$ with $j\in\{0,1728\}$ might change
even if $t$ remains the same, contrary to the gist of (\ref{eq:q0act}).
For example, if $t=-1$ then $t=1/t$ but $Q_0\mapsto -Q_0$.
The action on $Q_0$ is then determined not only by the action on $t$,
but also by the permutation of $e_0^2,e_1^2,e_t^2,e_{\infty}^2$.

As it turns out, complications with $k_1,k_2$ for $j\in\{0,1728\}$ do not have to be considered either,
partly because the encountered Heun equations with $j\in\{0,1728\}$ have some
exponent differences equal. The ambiguous case (\ref{eq:nodsing}) is bound to happen
for most other encountered Heun equations, because they have a free parameter.


The invariants can be expressed in terms of the parameters $a,b,c,d$ rather then
the exponent differences $e_0,e_1,e_t,e_\infty$. For example,
the invariant $Q_1$ can be computed using
\begin{eqnarray} \label{eq:invq}
Q_0\equal 12q-6ab+(a^2+b^2-2cd+2c+2d-1)(2t-1) \nonumber\\
&& -(c^2+2ad+2bd)(t-2)-(d^2+2ac+2bc)(t+1).
\end{eqnarray}


\begin{table} \small
\begin{tabular}{@{}llll@{}}
\hline 
Id & Invariant $Q_1$ &  Invariant $k_1$ & Invariant $k_2$ \\ %
\hline
\PR{25} & $\frac{3^5\cdot7}{2^5\,5^2}(72\alpha^2-1)$
 & $\frac73 \left(81\la^4-\frac{173}{16}\la^2+\frac14\right)$
 & $\frac{100}{189} \left(81\la^4+\frac{121}{16}\la^2+\frac14\right)$ \\
\PR{26} & $\frac{7\cdot 13\cdot 17\cdot 37\cdot 41}{2^7\,3^6\,5^4}(1881\la^2-28)$
 & $-\frac{481}{400}\left( 25\la^4-\frac{229}9\la^2+\frac{4}{81} \right)$ 
 & $\frac{81}{481}\left( 25\la^4+\frac{70549}{729}\la^2+\frac{4}{81} \right)$  \\
\PR{27} & $\frac{7^2\cdot 23\cdot 41\cdot 79\cdot 6481}{2^{11}\,3^8\,5^2}(1696\la^2-83)$
 & $-\frac{6481}{81} \! \left( 36\la^4-\frac{643}{256}\la^2+\frac{1}{64} \right)$ 
 & $\frac{6400}{6481}\left( 36\la^4+\frac{52493}{20480}\la^2+\frac{1}{64} \right)$  \\
\PR{28} & $\frac{7\cdot 23\cdot 41\cdot 79\cdot 6481}{2^7\,3^{10}\,5^2}(22977\la^2-440)$
 & $\frac{6481}{144}\la^2 \left(691\la^2-\frac{43}9\right)$
 & $\frac{1}{6481}\la^2 \left(2726825\la^2+\frac{166577}9\right)$ \\
\PR{29} & $\frac{2^2\cdot 7\cdot 13\cdot 23\cdot 29\cdot 97}{3^8\,5^4}(1215{\la}^2-16)$
 & $\frac{679}9\la^2 \left(17\la^2-\frac{404}{225}\right)$
  & $\frac{20}{679}\la^2\!\left(3775\la^2+\frac{3274}9\right)$ \\
\PR{30} & $\frac{5^2\cdot 11\cdot 13\cdot 29\cdot 757}{2^5\,3^8\,7^2}(2952\la^2-71)$
 & $\frac{757}{27} \left(144\la^4-\frac{715}{144}\la^2+\frac1{36}\right)$
  & $\frac{784}{757} \left(144\la^4+\frac{131365}{28224}\la^2+\frac1{36}\right)$ \\
\PR{31} & $\frac{7\cdot 11\cdot 37\cdot 59\cdot 127}{2^{11}\,3^8\,5^2}(14004\la^2-323)$
 & $-\frac{889}{216} \left(4\la^4-\frac{52}{9}\la^2+\frac1{36}\right)$
  & $\frac{100}{889} \left(4\la^4+\frac{970}{9}\la^2+\frac1{36}\right)$ \\
\PR{32} & $\frac{7\cdot 11\cdot 37\cdot 59\cdot 127}{2^9\,3^{8}\,5^2}(4014\alpha^2-95)$ 
 & $889\la^2\left(\la^2-\frac1{162}\right)$ & $\frac{36}{889}\la^2\left(8450\la^2+\frac{5641}{162}\right)$ \\
\PR{33} & $\frac{7\cdot 11\cdot 23\cdot 61\cdot 131\cdot 2287}{2^{15}\,3^2\,5^6}(12580\la^2-237)$
 & $\frac{16009}{40}\la^2\left(23\la^2-\frac{63}{100}\right)$
 & $\frac{250}{16009}\la^2\left(15309\la^2+\frac{40709}{100}\right)$ \\
\PR{34} & $\frac{7\cdot 11\cdot 23\cdot 61\cdot 131\cdot 2287}{2^{13}\,3^2\,5^8}(31450\la^2\!-\!1953)$
 & $-\frac{16009}{100}\la^2\left(7\la^2-\frac{3}{100}\right)$
 & $\frac{8}{16009}\la^2\left(88749\la^2+\frac{40621}{100}\right)$ \\
\PR{35} & $\frac{5\cdot 13\cdot 17\cdot 113\cdot 4993}{2^9\,3^{10}\,7^4}(8262{\la}^2-67)$
 & $\frac{4993}9\la^2\!\left(\la^2-\frac{17}{882}\right)$
 & $\frac{2}{4993}\la^2\!\left(672868\la^2+\frac{64153}9\right)$ \\
\PR{36} & $\frac{11\cdot 19\cdot 23\cdot 61\cdot 157\cdot 1459}{2^{13}\,3^8\,5^6\,7^2}(50184{\la}^2\!-\!635)$
 & $\frac{27721}{8000} \left(100\la^4-\frac{1849}{324}\la^2+\frac{1}{36}\right)$
 & $\frac{35721}{27721} \! \left(100\la^4+\frac{4322681}{1285956}\la^2+\frac{1}{36}\right)$ \\
\PR{37} & $\frac{3^5\cdot 7\cdot 11}{2^5}(4\alpha^2-1)$ 
 & $-33\left(\la^2-\frac14\right)^{2}$ & not needed \\
\PR{38} & $\frac{3^3\cdot 5\cdot 11}{2^4}(4\la^2-1)$
 & $\quad\,15\left(\la^2-\frac14\right)^{2}$ & not needed \\
\PR{39} & $\frac{7\cdot 193\cdot 383}{2^8\cdot 3}(176\la^2-71)$
 & $-193\left(\la^2-\frac1{16}\right)^{2}$ & not needed \\
\PR{40} & $\frac{5\cdot 53\cdot 109}{3^5}(63\la^2-23)$
 & $\qquad 53\left(\la^2-\frac19\right)^{2}$ & not needed \\
\PR{42} & $\frac{3\cdot 5\cdot 17\cdot 31}{2^5}(12\la^2-1)$
 & $-51\la^2\left(4\la^2-\frac14\right)$ & not needed \\
\PR{43}  & $\frac{2^5\cdot11\cdot 31\cdot 251}{5^3}(58\la^2-1)$ & $\;\;\,71424\la^4$ & not needed \\
\PR{44} & $\frac{3^5\cdot 13\cdot 17}{2^8\,5^4}(1-16\la^2)$ 
 & $39\left(\la^2-\frac1{16}\right)\left(\la^2-\frac1{400}\right)$ & not needed \\
\PR{45} & $\frac{11\cdot 109\cdot 157}{2^4\,5^6}(5-356\la^2)$
 & $-\frac{109}5\!\left(\la^2-\frac1{4}\right)\! \left(\la^2-\frac1{100}\right)$ & not needed \\
\PR{46} & $\frac{3^2\cdot 11\cdot 109\cdot 157}{2^2\,5^8}(4-75\la^2)$
 & $\frac{327}{25}\la^2\left(\la^2-\frac1{25}\right)$ & not needed \\
\PR{47} & $\frac{5\cdot 19\cdot 43}{2^3\,3^8}(11-36\la^2)$
 & $\frac{95}{972} \left(\la^2-\frac1{4}\right)$ & not needed \\
\PR{48} & $\frac{7\cdot 47\cdot 337}{2^7\,3^{8}}(44\la^2-1)$
 & $\frac{47}9 \left(\la^2-\frac1{4}\right) \left(\la^2-\frac1{36}\right)$ & not needed \\
\PR{49} & $\frac{7\cdot 47\cdot 337}{2^5\,3^{10}}(2-513\la^2)$
 & $-\frac{329}9\la^2\left(\la^2-\frac1{9}\right)$ & not needed \\ 
\PR{50} & $\frac{2\cdot 5\cdot 11\cdot 239\cdot 251}{3^{12}}(4-27{\la}^2)$
 & $\frac{239}{81}\la^2 \left(\la^2-\frac1{9}\right)$ & not needed \\
\PR{56} & $\frac{2\cdot 5\cdot 11\cdot 13\cdot 103}{3^5\,7^4}(1-153\la^2)$
 & $-22\left(\la^2-\frac1{9}\right)\! \left(\la^2-\frac1{441}\right)$ & not needed \\
\PR{57} & $\frac{3\cdot 5\cdot 11\cdot 13\cdot 37\cdot 43}{2^{15}\,7^3}(1-244\la^2)$
 & $-\frac{645}8\la^2\left(\la^2-\frac1{4}\right)$ & not needed \\
\PR{58} & $\frac{5\cdot 19\cdot 269\cdot 499}{2^{19}\,3^5}(68\la^2+3)$
 & $\frac{269}{64}\la^2\left(\la^2-\frac1{4}\right)$ & not needed \\
\PR{59} & $\frac{11\cdot 71\cdot 167}{2^{13}\,3^5\cdot 5}(17-72\la^2)$
 & $\frac{355}{2304}\left(\la^2-\frac1{9}\right)$ & not needed \\
\PR{60} & $\frac{229-53i}{2\cdot 5^2} \left( \frac{4-i}{16}-\left(10-7i \right)\la^2 \right)$
 & $\frac{2+3i}{2}\left(16\la^4-\frac{28+9i}{8}\la^2+\frac1{64}\right)$
 & $\frac{17-6i}{13}\left(16\la^4+\frac{9+12i}{5}\la^2+\frac1{64}\right)$  \\
\PR{61} & $\frac{531-6130\omega}
{2^33^3(2-\omega)^2} \left( \frac{5-4\omega}{36}\!-\!(19\!-\!8\omega)\alpha^2\right)$ 
 & $\frac{11\omega-23}{9}\!\left(9\la^4\!+\!\frac{2-21\omega}{4+8\omega}\la^2\!+\!\frac1{144}\!\right)$
 & $\frac{12-20\omega}{15+19\omega}\!\left(9\la^4\!-\!\frac{157+34\omega}{10+16\omega}\la^2\!+\!\frac1{144}\!\right)$ 
\\ \hline
\end{tabular}  \centering
\caption{Invariants for sufficient identification of reducible Heun equations.}  \label{heunq}
\end{table}

The following theorem formulates the sufficient conditions for Theorem \ref{th:t}.
\begin{theorem}\label{th:tinv}
Heun's equation $(\ref{Heun})$ is (a specialization of a) parametric
pull-back transformation of a hypergeometric equation if it satisfies
one of the conditions \refpart{i}--\refpart{v} of Theorem $\ref{th:t}$,
and the following respective conditions:
\begin{itemize}
\item[\refpart{i}] $Q_0=0$, and $k_1=3(\la^2-\lb^2)(\la^2-\lc^2)$ for $\heunde{\la,\la,\lb,\lc}$;
\item[\refpart{ii}] $Q_0=0$;
\item[\refpart{iii}] $Q_1=\frac{5\cdot 7\cdot 13}{2^5\,3^2}(4\la^2+8\beta^2-3)$,
$k_1=-13\,(\la^4-\frac54\la^2\lb^2+\frac1{16}\lb^2)$,
$k_2=\frac{36}{13}\la^4+5\la^2\lb^2+\frac5{13}\la^2+\frac9{52}\lb^2$;
\item[\refpart{iv}]  $Q_1=\frac{5\cdot 7\cdot 17\cdot 73}{2^5\,3^4}(23{\la}^2+23{\lb}^2-4)$,
$k_1=\frac{73}2\,(\la^4-4\la^2\lb^2+\lb^4)$,
$k_2=\frac{324}{73}\,(\la^4+\frac{274}{81}\la^2\lb^2+\lb^4)$ \\
for $\heunde{\la,2\la,\lb,2\lb}$, and \\
$Q_1=\frac{5\cdot 7\cdot 17\cdot 73}{2^4\,3^4}(23{\la}^2+23{\lb}^2-2)$,
$k_1=-73\,(\la^4-10\la^2\lb^2+\lb^4)$,
\mbox{$k_2=\frac{576}{73}\,(\la^4+\frac{185}{16}\la^2\lb^2+\lb^4)$} \\
for $\heunde{\la,3\la,\lb,3\lb}$;
\item[\refpart{v}]
for the cases \PR{25} to \PR{36}, the invariants $Q_1,k_1,k_2$ are as in Table $\ref{heunq}$;\\
for the cases \PR{37}--\PR{40}, \PR{42}--\PR{50}, \PR{56}--\PR{59},
 the invariants $Q_1,k_1$ are as in Table $\ref{heunq}$;
\item[\refpart{vi}]
the invariants $Q_1,k_1,k_2$ are as in Table $\ref{heunq}$,
or conjugated $i\mapsto -i$, $\omega\mapsto-\omega-1$ if the $j$-invariant
of Table $\ref{heunred}$ is conjugated.
\end{itemize}
\end{theorem}
\begin{proof}
If $j=0$ as in \refpart{ii}, the invariants $Q_1,k_1,k_2$ generally fail.
But for the encountered Heun equations with $j=0$, part \refpart{b} of Theorem \ref{th:invs} applies,
and the semi-invariant value $Q_0=0$ determines the accessory parameter.

If $j=1728$ as in \refpart{i}, the semi-invariant value $Q_0=0$ determines the accessory parameter
just as well. The invariant $k_1$ has only two possible values: $3(\la^2-\lb^2)(\la^2-\lc^2)$ and
$-\frac32(\la^2-\lb^2)(\la^2-\lc^2)$. The latter $k_1$-value gives a confusion between two $t$-values in
$\{-1,2,1/2\}$, but the encountered Heun equations have the former $k_1$-value. This $k_1$-value
gives an equation of the form $(t-\xi)^2=0$ and determines the correct $t\in\{-1,2,1/2\}$
without the aid from $k_2$.

In case \refpart{iv}, the two different Heun equations have to be considered separately.
Note that the transformation P24 is a specialization of both P19 and P20,
and the invariants specialize consistently to
$Q_1=\frac{5\cdot 7\cdot 17\cdot 73}{\;2^4\,3^4}(115{\la}^2-2)$,
$k_1=1679\la^4$, $k_2=\frac{36432}{73}\la^4$.

In the cases \PR{37}--\PR{40}, \PR{42}--\PR{50}, \PR{56}--\PR{59}, we have two equal 
exponent differences. The $k_1$-invariant gives then
ambiguity  only for $t\in\{-1,2,1/2\}$, while the actual $t$-values are algebraic.
Hence the invariant $k_2$ is not needed.

In all other cases, the full invariants set $Q_1,k_1,k_2$ is used.
\end{proof}

\section*{Acknowledgements}
The authors are very grateful to Robert S. Maier for sharing his knowledge of
literature and ongoing developments related to the subject of this article, and
a coordination discussion.

The first author is supported by the JSPS grant No 20740075.
Some of the calculations by the second author (GF) were partially obtained in the
Inter\-discip\-li\-na\-ry Centre for Mathematical and
Computational Modelling (ICM), War\-saw Uni\-ver\-si\-ty, wi\-thin
grant nr G34-18. Research of the second author  is partially
supported by Polish MNiSzW Grant No N N201 397937. The authors RV
and GF are grateful to  the
 organizers of the XVth Conference on Analytic Functions
 and Related Topics held in Chelm in July 2009 for the
 hospitality.


\end{document}